\documentclass[11pt,reqno,british,a4paper,twoside]{article}

\usepackage[utf8]{inputenc}
\usepackage[sc]{mathpazo}
\usepackage{babel,mathtools,amsmath,amssymb,url,paralist,xspace,mathrsfs}
\usepackage{booktabs,microtype,array,fixltx2e,fancyhdr,amscd,csquotes,tikz}

\usetikzlibrary{matrix}

\usepackage[margin=30pt,font=small,labelfont=bf,labelsep=period]{caption}
\usepackage[style=numeric,backend=biber,maxbibnames=10]{biblatex}
\usepackage[amsmath,thmmarks]{ntheorem}
\usepackage[capitalise]{cleveref}
\rmfamily
\mathtoolsset{mathic=true}
\setdefaultenum{\upshape (i)}{}{}{}

\renewcommand{\bibfont}{\small}

\theoremstyle{plain}
\theoremseparator{.}
\newtheorem{theorem}{Theorem}[section]
\newtheorem{proposition}[theorem]{Proposition}
\newtheorem{lemma}[theorem]{Lemma}
\newtheorem{corollary}[theorem]{Corollary}

\theorembodyfont{\normalfont}
\newtheorem{definition}[theorem]{Definition}

\theoremheaderfont{\itshape}
\theoremsymbol{\begin{small}\ensuremath{\quad\triangle}\end{small}}
\newtheorem{remark}[theorem]{Remark}
\theoremsymbol{\begin{small}\ensuremath{\quad\diamondsuit}\end{small}}
\newtheorem{example}[theorem]{Example}

\theoremstyle{nonumberplain}
\theoremheaderfont{\itshape}
\theorembodyfont{\normalfont}
\theoremsymbol{\begin{small}\ensuremath{\quad\Box}\end{small}}
\newtheorem{proof}{Proof}

\qedsymbol{\begin{small}\ensuremath{\quad\Box}\end{small}}

\numberwithin{equation}{section}
\numberwithin{table}{section}


\newcolumntype{C}{>{$}c<{$}}
\newcolumntype{L}{>{$}l<{$}}
\newcolumntype{R}{>{$}r<{$}}

\let\oldbibliography\thebibliography
\renewcommand{\thebibliography}[1]{%
\oldbibliography{#1}%
\small
\setlength{\itemsep}{0pt}%
\setlength{\parskip}{0pt plus 1pt}%
}

\newcommand{\Lie}[1]{\operatorname{\textsl{#1}}}
\newcommand{\lie}[1]{\operatorname{\mathfrak{#1}}}
\newcommand{\Lel}[1]{{\mathsf{#1}}}  
\newcommand{\GL}{\Lie{GL}}
\newcommand{\PSU}{\Lie{PSU}}
\newcommand{\SL}{\Lie{SL}}
\newcommand{\SO}{\Lie{SO}}
\newcommand{\Sp}{\Lie{Sp}}
\newcommand{\SU}{\Lie{SU}}

\newcommand{\Spin}{\Lie{Spin}}
\newcommand{\la}{\lie{a}}
\newcommand{\g}{\lie{g}}
\newcommand{\h}{\lie{h}}
\newcommand{\lk}{\lie{k}}
\newcommand{\n}{\lie{n}}
\newcommand{\lr}{\lie{r}}

\newcommand{\p}{\lie{p}}

\newcommand{\su}{\lie{su}}
\newcommand{\lt}{\lie{t}}

\newcommand{\Ld}{\mathcal L}
\newcommand{\Ldc}{(\hook\mathcal L)}
\newcommand{\Pg}[1][\g]{\mathcal P_{#1}}
\newcommand{\Pgk}{\Pg[\g,k]}
\newcommand{\dP}{d_{\mathcal P}}
\newcommand{\X}{\mathcal X}

\newcommand{\Hodge}{{*}}

\newcommand{\bC}{{\mathbb C}}
\newcommand{\bH}{{\mathbb H}}
\newcommand{\bR}{{\mathbb R}}
\newcommand{\bZ}{{\mathbb Z}}

\newcommand{\fX}{\lie X}

\DeclareMathOperator{\ad}{ad}
\DeclareMathOperator{\Ad}{Ad}
\DeclareMathOperator{\codim}{codim}
\DeclareMathOperator{\coker}{coker}

\DeclareMathOperator{\End}{End}

\DeclareMathOperator{\im}{im}

\DeclareMathOperator{\stab}{stab}
\DeclareMathOperator{\vol}{vol}

\newcommand{\Cyclic}{\mathop{\text{\Large$\mathfrak S$}\vrule width 0pt
depth 2pt}}

\newcommand{\hook}{{\lrcorner\,}}

\newcommand{\NB}{\nabla}
\newcommand{\LC}{\NB^{\textup{LC}}}

\newcommand{\sslash}{\mathbin{{/}\mkern-6mu{/}\mkern-2mu}}
\newcommand{\mred}[1][\nu]{\sslash_{#1}}

\DeclarePairedDelimiter{\norm}{\lVert}{\rVert}
\DeclarePairedDelimiter{\Span}{\langle}{\rangle}

\DeclarePairedDelimiter{\inpd}{\langle}{\rangle}
\newcommand{\inp}[2]{\inpd{#1,#2}}

\newcommand{\eqbreak}{\\&\qquad}

\makeatletter
\newcommand{\mythanks}{\xdef\@thefnmark{}\@footnotetext}
\makeatother

\begin{document}

\thispagestyle{empty}

\begin{center}
  \LARGE\bfseries Closed forms and multi-moment maps
\end{center}
\begin{center}
  \Large Thomas Bruun Madsen and Andrew Swann
\end{center}


\begin{abstract}
  We extend the notion of multi-moment map to geometries defined by
  closed forms of arbitrary degree.  We give fundamental existence and
  uniqueness results and discuss a number of essential examples,
  including geometries related to special holonomy.  For forms of
  degree four, multi-moment maps are guaranteed to exist and are
  unique when the symmetry group is \( (3,4) \)-trivial, meaning that
  the group is connected and the third and fourth Lie algebra Betti
  numbers vanish.  We give a structural description of some classes of
  \( (3,4) \)-trivial algebras and provide a number of examples.
\end{abstract}

\bigskip
\begin{center}
  \begin{minipage}{0.8\linewidth}
    \begin{small}
      \tableofcontents
    \end{small}
  \end{minipage}
\end{center}

\bigskip
\begin{small}\noindent
\end{small}
\newpage

\section{Introduction}
\label{sec:introduction}

The rich and varied field of symplectic geometry is the study of
closed non-degenerate two-forms.  It has origins in the study of
Hamiltonian dynamics and the geometry of phase space.  From a
mathematical point of view it is natural to try to see how much of
this theory may be extended to closed forms of higher degree.  A
number of authors have already made attempts at generalising the
Hamiltonian picture to higher-degree, or multi-, phase spaces, often
motivated by the interest in various field theories
\cite{Carinena-CGM:reduction,Carinena-CI:multisymplectic,Gotay-IMM:covariant,Baez-HR:string,Baez-R:2-algebra}.
Indeed string- and \( M \)-theories with fluxes give a number of
geometries equipped with closed differential forms of varying degrees,
see \cite{Gran-G-P:blackh5fI} for one such example.

The purpose of this article is to study the geometry of closed
differential forms in general, with particular emphasis on new
techniques that are available in the presence of symmetry.  One main
tool in the construction of various symplectic manifolds is the
Marsden-Weinstein quotient formed by taking quotients of the level
sets of a moment map.  One important feature of the moment map in
symplectic geometry, is that it takes values in a finite-dimensional
vector space depending only on the symmetry group and not on the
underlying manifold.  Previous attempts to extend moment maps to forms
of higher degrees, have produced maps taking values in
infinite-dimensional spaces of forms over the manifold, see the
references above, though \cite{Neeb-C:Hamiltonian} provides an
interesting alternative.  In \cite{Madsen-S:2-3-trivial} we
introduced a new notion of multi-moment map for geometries with a
closed three-form, which shares the above basic property of symplectic
moment maps.  A thorough study of these new maps was made
in~\cite{Madsen-S:multi-moment}.  In this paper we will show how this
theory extends to forms of arbitrary degree, in large part based on
ideas developed in the thesis \cite{Madsen:mm-thesis}.  Not only do
these multi-moment maps take values in a finite-dimensional vector
space, but there are existence results based on easily satisfied
properties of the manifold or its symmetry group.  We will thus
describe the general theory, give examples of multi-moment reduction
of various geometries, particular ones with a closed four-form, and
study an algebraic condition on Lie groups that guarantees existence
and uniqueness of multi-moment maps for four-form geometries.

One salient feature of symplectic geometry is that the two-form is
non-degenerate.  What this means for a form of higher degree is less
clear and we start the paper in \cref{sec:distinguished} by discussing
a number of different possibilities.  These distinguish a number of
geometries that have importance in their own right, for example
geometries with exceptional holonomy, but do not lead to any one good
constraint, so for the general theory we do not impose such
assumptions.

In \cref{sec:multi-moment} we then introduce the notion of
multi-moment map for symmetries of closed geometries of arbitrary
degree.  In order to facilitate the proofs we develop some theory of
multi-vectors on manifolds and in particular give an extension of the
classical Cartan formula which expresses the Lie derivative of forms
in terms of exterior derivatives and contractions.  Multi-moment maps
are then defined, and existence and uniqueness theorems proved under
topological and under algebraic assumptions.
\Cref{sec:example-geometries} then gives a number of examples of
closed geometries, computes multi-moment maps in a number of cases and
discusses the geometries of quotients.  Finally, in
\cref{sec:cohomology} we study the algebraic condition found in
\cref{sec:symmetries} for the existence and uniqueness of multi-moment
maps for geometries with a closed four-form.  These conditions are
expressed as the vanishing of the third and fourth Lie algebra
cohomology groups.  We show how to exploit the Hochschild-Serre
spectral sequence to determine the algebraic structure of a wide class
of such Lie algebras and give a number of examples.

\paragraph*{Acknowledgements}
We gratefully acknowledge financial support from \textsc{ctqm},
\textsc{geomaps} and \textsc{opalgtopgeo}.  AFS is also partially
supported by the Danish Council for Independent Research, Natural
Sciences, Symmetry Techniques in Differential Geometry and by the
Ministry of Science and Innovation, Spain, under Project
\textsc{mtm}{\small 2008-01386}.  AFS thanks the organisers of the
\textsc{gesta} meeting 2011 for a stimulating event and the
opportunity to present aspects of this material.

\section{Distinguished differential forms}
\label{sec:distinguished}

We will be considering geometries defined by closed differential
forms.  So as a first question we address the issue of whether there
are any algebraically distinguished forms on a vector space.  Recall
that in symplectic geometry one makes repeated use of the
\enquote{non-degeneracy} of the symplectic two-form~\( \omega \).
Algebraically this leads to the fact that a symplectic manifold is of
even dimension and then the closure of \( \omega \) is used for
Darboux's Theorem, that there are local coordinates so that \( \omega
= dx_1\wedge dy_1 +\dots +dx_n\wedge dy_n \).  For higher degree
forms, the situation is not simple and it is not clear which
definition is appropriate.  Let us discuss some of the possibilities.

Let \( V \) be an \( n \)-dimensional vector space over \( \bR \).
Write \( \Lambda^*V^* \) for the algebra of forms on \( V \).

\begin{definition}
  A form \( \alpha \in \Lambda^rV^* \) is said to be \emph{fully
  non-degenerate} if
  \begin{equation*}
    \alpha(v_1,v_2,\dots,v_{r-1},\cdot) 
  \end{equation*}
  is non-zero whenever \( v_1,\dots,v_{r-1} \in V \) are linearly
  independent.
\end{definition}

For \( r=2 \), this is the usual non-degeneracy of a two-form.  For
any two-form \( \alpha \) there is a basis of \( V^* \) such that
\begin{equation}
  \label{eq:deg2}
  \alpha = e_1\wedge e_2 + \dots + e_{2k-1} \wedge e_{2k},
\end{equation}
for some \( k \leqslant \tfrac12\dim V \).  To see this start with a
non-zero vector \( X \in V \) and put \( e_2 = \alpha(X,\cdot) \).
Choose a vector \( Y \) such that \( e_2(Y) = 1 \) then choose a
one-form \( e_1 \) with \( e_1(X) = 1 \) and \( e_2(Y) = 0 \).  We now
have \( \alpha' = \alpha - e_1\wedge e_2 \) is zero on \( X \) and \(
Y \), and the result follows by induction.  We see that \( \alpha \)
is non-degenerate if and only if \( \dim V = 2k \).

For forms of degree \( 3 \), full non-degeneracy already gives much
stronger restrictions.

\begin{theorem}
  \label{thm:fully}
  A vector space of dimension \( n \) admits a fully non-degenerate of
  form of degree \( r \geqslant 3 \) if and only if \( r=n \) or the
  pair \( (r,n) \) is either \( (3,7) \) or \( (4,8) \).
\end{theorem}

\begin{proof}
  Choose an inner product \( \inp\cdot\cdot \) on \( V \).  A form \(
  \alpha \in \Lambda^rV^* \) defines a cross-product like operation \(
  V^{r-1} \to V \) via
  \begin{equation*}
    \inp{v_1 \times v_2 \times \dots \times v_{r-1}}{w} =
    \alpha(v_1,v_2,\dots,v_{r-1},w).
  \end{equation*}
  This operation is continuous and has the property that the product
  \( v_1\times\dots\times v_{r-1} \) is orthogonal to each of the \(
  v_i \).  When \( \alpha \) is fully non-degenerate, this product on
  linearly independent vectors is non-zero.  Let \( V_{r,n} \) denote
  the Stiefel manifold consisting of all \( r \)-tuples \(
  (f_1,\dots,f_r) \) of orthonormal vectors in \( \bR^n \).  The map
  \begin{equation*}
    (f_1,\dots,f_{r-1}) \mapsto \left(f_1, \dots, f_{r-1}, \frac{f_1
      \times\dots\times f_{r-1}}{\norm{f_1 \times\dots\times
      f_{r-1}}}\right)   
  \end{equation*}
  is a cross section of the projection \( V_{r,n} \to V_{r-1,n} \).
  It is a topological result of \textcite{Whitehead:Stiefel} that such
  a cross-section exists only in the given cases.  An elementary proof
  for the case of two-fold cross-products, \( r=3 \), may be found in
  \cite{Massey:cross-products}.
\end{proof}

For \( r=n \), a volume form on \( V \) provides a fully
non-degenerate form.  Examples for the other two cases of this result
are given by the three-forms
\begin{equation}
  \label{eq:G2}
  \phi_0 = e_{123} + e_{145} + e_{167} + e_{246} - e_{257} - e_{347}
  - e_{356}
\end{equation}
on \( \bR^7 \) and the four-form
\begin{equation}
  \label{eq:Spin7}
  \begin{split}
    \Phi_0 &= e_{1234} + e_{1256} + e_{3478} + e_{3456} + e_{1278} +
    e_{1357} - e_{1368} \eqbreak - e_{2457} + e_{2468} - e_{1458} -
    e_{1467} - e_{2358} - e_{2367} + e_{5678}
  \end{split}
\end{equation}
on \( \bR^8 \).  Here \( e_1,\dots,e_n \) is a basis for \( (\bR^n)^*
\) and wedge products have been omitted from the notation, so \(
e_{123} = e_1\wedge e_2\wedge e_3 \), etc.

The forms \( \phi_0 \) and \( \Phi_0 \) have interesting geometric
properties.  In particular, if we consider the action of \( \GL(n,\bR)
\) then the isotropy groups \( \{ g\in\GL(n,\bR) : g\cdot \alpha =
\alpha \} \) are the compact 14-dimensional exceptional Lie group \(
G_2 \) for \( \alpha = \phi_0 \) and the compact 21-dimensional group
\( \Spin(7) \), the simply-connected double cover of \( \SO(7) \), for
\( \alpha = \Phi_0 \), see \textcite{Bryant:exceptional}.  We now see
that the dimensions of the orbits of these forms are
\begin{gather*}
  \dim(\GL(7,\bR)\cdot\phi_0) = 49 - 14 = 35,\\
  \dim(\GL(8,\bR)\cdot\Phi_0) = 64 - 21 = 43.
\end{gather*}
The first of these is notable since \( \dim\Lambda^3\bR^7 = 35 \), so
the orbit of \( \phi_0 \) in \( \Lambda^3\bR^7 \) is open.

\begin{definition}[\textcite{Hitchin:forms}]
  \label{def:stable}
  A form \( \alpha \in \Lambda^rV^* \) is \emph{stable} if the orbit
  \( {\GL(V)\cdot\alpha} \) is open in \( \Lambda^rV^* \).
\end{definition}

For general forms the condition of stability provides restrictions on
the dimension of \( V \).

\begin{proposition}
  A vector space of dimension \( n \) admits a stable form of degree
  \( r \) if and only if either \( r\in\{1,2,n-2,n-1,n\} \) or \( r
  \in \{3,n-3\} \) with \( n \in \{6,7,8\} \).
\end{proposition}

\begin{proof}
  We give the basic arguments, following \textcite{Hitchin:forms}.

  The dimension of the orbit \( \GL(n,\bR)\cdot\alpha \) is at most \(
  \dim(\GL(n,\bR)) = n^2 \).  To have a stable form we thus need \(
  n^2 \geqslant \dim\Lambda^r\bR^n = \binom nr \).  The binomial
  coefficient \( \binom nr \) is a polynomial of degree \( r \) in \(
  n \), which for \( 3 \leqslant r \leqslant n-3 \) grows quicker than
  \( n^2 \).  Now for \( r<n/2 \), we have \( \dim \Lambda^r\bR^n < \dim
  \Lambda^{r+1}\bR^n \), so we start by considering the case \( r = 3
  \).  In this case, we see that
  \begin{equation*}
    \begin{split}
      \dim\Lambda^3\bR^n - \dim(\GL(n,\bR))
      &= \tfrac16n(n-1)(n-2) - n^2 \\
      &= \tfrac16 n((n-9)n + 2)
    \end{split}
  \end{equation*}
  so an orbit in \( \Lambda^3\bR^n \) can not be open if \( n
  \geqslant 9 \).  In dimension \( n = 8 \), we have \(
  \dim\Lambda^3\bR^8 < 64 = \dim\GL(8,\bR)\), but \(
  \dim\Lambda^4\bR^8 = 70 > 64 \), so orbits of four-forms on \( \bR^8
  \) are never open.  This gives the list of possible \( r \) and \( n
  \) in the statement.

  It remains to show that each case can be realised.  For \( r=1,n-1,n
  \), we take \( \alpha \) to be any non-zero form of the given
  degree.  For \( r=2 \), open orbits are realised by forms as in
  \eqref{eq:deg2} with \( k = \lfloor n/2 \rfloor \).  Taking the
  Hodge star of such a two-form gives a stable form of degree \( n-2
  \).

  Finally, we need to give appropriate three-forms in dimensions \( 6
  \), \( 7 \) and \( 8 \); the case for \( r = n-3 \) will then follow
  by taking Hodge stars.  For dimension \( n = 6 \), one can take \(
  \alpha \) to be the real part of a complex volume form on \( \bR^6 =
  \bC^3 \).  We have already seen \( \phi_0 \) \eqref{eq:G2} is stable
  on \( \bR^7 \).  Finally for \( n=8 \), one identifies \( \bR^8 \)
  with the Lie algebra \( \su(3) \).  This carries an \( \ad
  \)-invariant three-form \( \alpha(X,Y,Z) = \inp{[X,Y]}Z \), which in
  an appropriate basis is
  \begin{equation}
    \label{eq:PSU3}
    \begin{split}
      \rho_0 = e_{123} &+ \frac12 e_1(e_{47}-e_{56})+\frac12
      e_2(e_{46}+e_{57}) \eqbreak + \frac12 e_3(e_{45}-e_{67}) +
      \frac{\sqrt3}2 e_8(e_{45}+e_{67}).
    \end{split}
  \end{equation}
  The infinitesimal stabiliser of this form is \( \su(3) \) and so the
  orbit of \( \rho_0 \) has dimension \( 64 - 8 = 56 =
  \dim\Lambda^3\bR^8 \) and is open.
\end{proof}

\noindent
Note that the connected subgroup of \( \GL(8,\bR) \) preserving \(
\rho_0 \) is~\( \PSU(3) \): the quotient of \( \SU(3) \) by its centre
\( \bZ/3 \).

So far we have considered two strong conditions on forms and found
them to be rather restrictive.  There is another condition that is
useful more generally.

\begin{definition}
  A form \( \alpha \) on \( V \) is \emph{(weakly) non-degenerate} if
  \begin{equation*}
    v\hook \alpha = \alpha(v,\cdot,\dots,\cdot)
  \end{equation*}
  is non-zero for each non-zero \( v \).
\end{definition}

Any non-zero form \( \alpha \) gives rise to a non-degenerate form on
the quotient \( V/\ker\alpha \) where \( \ker\alpha = \{v\in V:
v\hook\alpha = 0\} \).  Conversely a volume form always provides a
non-degenerate form on any vector space.  For a particular degree of
form there can be restrictions on the dimension.  For two-forms weak
and full non-degeneracy are the same and the space must be
even-dimensional.  In higher degree we have far fewer restrictions.

\begin{proposition}
  \label{prop:non-degenerate}
  A vector space of dimension \( n \) admits a non-degenerate form of
  degree \( r \) with \( r \geqslant 3 \) if and only if \( n
  \geqslant r \) and \( n \ne r+1 \).
\end{proposition}

\begin{proof}
  For \( n < r \), we have \( \Lambda^rV^*=\{0\} \), so all \( r
  \)-forms are zero and thus degenerate.  For \( n = r+1 \), any form
  of degree \( r \) is the Hodge dual of a one-form and so has the
  form \( \alpha = e_2\wedge\dots\wedge e_n \), which is degenerate.

  To demonstrate existence of non-degenerate forms in the remaining
  cases, first consider \( r=3 \).  If \( n \geqslant 3 \) is odd, let
  \( \omega \) be a non-degenerate two-form on \( \bR^{n-1} \), then
  \( \alpha = \omega \wedge e_n \) is a non-degenerate three-form on
  \( \bR^n \).  If \( n \geqslant 6 \) is even, then writing \( \bR^n
  = \bR^3 \oplus \bR^{n-3} \) we have a non-degenerate three-form
  given by \( \alpha = e_{123} + \alpha' \), where \( e_{123} \) is a
  volume form on \( \bR^3 \) and \( \alpha' \) is non-degenerate on \(
  \bR^{n-3} \).

  Now for \( r>3 \), given a non-degenerate form \( \alpha' \) of
  degree \( r-1 \) on \( \bR^{n-1} \) we have that \( \alpha = \alpha'
  \wedge e_n \) is non-degenerate of degree \( r \) on \( \bR^n \).
\end{proof}

\section{Multi-moment maps in general degree}
\label{sec:multi-moment}

The general situation we wish to study is where a symmetry group \( G
\) acts on a manifold \( M \) preserving some closed form.

\begin{definition}
  For \( r\geqslant 2 \), a \emph{closed geometry of degree \( r \)}
  on a manifold \( M \) is choice of a closed differential form \(
  \alpha \in \Omega^r(M) \).
\end{definition}

Here \( \alpha \) closed means \( d\alpha = 0 \) in the exterior
algebra.  This implies that if \( \ker\alpha \) has constant dimension
then \( \mathcal D = \ker\alpha \) is integrable as a distribution.
It follows that \( \alpha \) induces a weakly non-degenerate closed
form on \( M/\mathcal D \) when this quotient is a manifold.  In
general we will not make any non-degeneracy assumptions on \( \alpha
\).  However, when needed, the following terminology will be useful.

\begin{definition}[\textcite{Baez-HR:string}]
  A \emph{\( k \)-plectic structure} is a closed geometry \(
  (M,\alpha) \) of degree \( r = k+1 \) with \( \alpha \) (weakly)
  non-degenerate.
\end{definition}

It is easy to give a couple of elementary examples.  Firstly, if \( M
= \bR^n \) is a vector space, then any constant coefficient form \(
\alpha \) on \( M \) is closed, and the discussion of the previous
section gives many \( k \)-plectic examples.  Of particular importance
are the forms \( \phi_0 \) \eqref{eq:G2}, on \( \bR^7 \), \( \Phi_0
\) \eqref{eq:Spin7} and \( \rho_0 \) \eqref{eq:PSU3}, both on \( \bR^8
\).

\begin{example}[Multi-phase space]
  \label{ex:multi-phase}
  Given any manifold \( N \) we may consider \( M = \Lambda^kT^*N
  \).  This carries a tautological form \( \beta \in \Omega^k(M) \)
  given by
  \begin{equation*}
    \beta_b(X_1,\dots,X_k) = b(\pi_*X_1,\dots,\pi_*X_k),
  \end{equation*}
  where \( \pi\colon M = \Lambda^kT^*N \to N \) is the bundle
  projection.  Defining
  \begin{equation*}
    \alpha = d\beta,
  \end{equation*}
  we get a closed \( (k+1) \)-form on \( M \) which turns out to be
  non-degenerate so \( (M,\alpha) \) is \( k \)-plectic.  To see this,
  choose local coordinates \( q^1,\dots,q^n \) on \( U \subseteq N \)
  and note that \( dq^I = dq^{i_1} \wedge \dots \wedge dq^{i_{r-1}}\)
  gives a basis for each fibre of \( \Lambda^kT^*U \subset M \).
  Let \( p_I \) be the corresponding fibre coordinates, then
  \begin{equation*}
    \beta = \sum_I p_Idq^I,\quad \alpha = \sum_I dp_I \wedge dq^I,
  \end{equation*}
  and non-degeneracy of \( \alpha \) is manifest.  This is the higher
  degree generalisation of the usual symplectic structure on phase
  space \( T^*N \).  It is the central example in mechanics and field
  theory based approaches to the geometry of closed forms.
\end{example}

In order to build multi-moment maps we need to construct closed
one-forms out of the group action and the closed \( r \)-form \(
\alpha \).  This involves contracting \( \alpha \) with linear
combinations of \( r-1 \) vector fields.  Here it is convenient to use
the notion of multi-vectors.

\subsection{Calculus of multi-vectors}
\label{sec:calculus}

Recall that a multi-vector \( p \) of degree \( s \) on \( M \) is a
sum \( p = \sum_{i=1}^k q_i \) of simple multi-vectors of the form
\begin{equation}
  \label{eq:qq}
  q = X_1 \wedge X_2 \wedge \dots \wedge X_s
\end{equation}
with \( X_j \) smooth vector fields on \( M \).  We will use
\begin{equation*}
  \fX^s(M) = \Gamma(\Lambda^sTM) 
\end{equation*}
to denote the space of degree \( s \) multi-vectors on~\( M \).  This
is dual to the space \( \Omega^s(M) \) of differential forms of the
same degree.  We write \( \hook \) for the partial evaluation map \(
\fX^s(M) \times \Omega^r(M) \to \Omega^{r-s}(M) \),
\begin{equation*}
  (q \hook \beta)(Y_1,\dots,Y_{r-s}) =
  \beta(X_1,X_2,\dots,X_s,Y_1,\dots,Y_{r-s}).
\end{equation*}

When we consider symmetries, we will have use for a generalisation of
Cartan's formula \( \Ld_X\alpha = d(X\hook\alpha) + X\hook d\alpha \)
for the Lie derivative.  To this end note that one may regard \(
\fX(M) \) as a vector space over \( \bR \) and form the exterior
powers \( \Lambda_{\bR}^s\fX(M) \).  These spaces are larger than \(
\fX^s(M) \) which is equal to the exterior product of \( \fX(M) \)
over \( C^\infty(M) \).  There is a natural \( \bR \)-linear
projection \( \Lambda_{\bR}^s\fX(M) \to \fX^s(M) \) given on
decomposable elements by
\begin{equation}
  \label{eq:Qq}
  Q = X_1 \curlywedge \dots \curlywedge X_s
  \mapsto q = X_1 \wedge \dots \wedge X_s,
\end{equation}
where \( \curlywedge \) denotes the wedge product over \( \bR \).  For
a such a \( Q \), we write
\begin{equation*}
  Q_{\wedge i} = (-1)^{i-1} X_1\curlywedge \dots \curlywedge
  \widehat{X_i} \curlywedge \dots \curlywedge X_s
\end{equation*}
and \( Q_{\wedge ij} = (Q_{\wedge i})_{\wedge j} \).  We define
\begin{gather}
  \label{eq:Ldc}
  \Ldc_Q\alpha = \sum_{i=1}^s Q_{\wedge i} \hook \Ld_{X_i}\alpha
  \quad\text{and}\quad \\
  \label{eq:L}
  L(Q) = \sum_{1\leqslant i < j\leqslant s} [X_i,X_j] \curlywedge
  Q_{\wedge ij},
\end{gather}
and extend \( \bR \)-linearly to \( \Lambda_\bR^s\fX(M) \).

\begin{lemma}[Extended Cartan Formula]
  \label{lem:g-Cartan}
  For \( \alpha \in \Omega^r(M) \) and \( p \in \fX^s(M) \), we have
  \begin{equation*}
    p\hook d\alpha  - (-1)^s d(p\hook \alpha)
    = \Ldc^{}_P\alpha - L(P) \hook \alpha
  \end{equation*}
  for any \( P \in \Lambda_{\bR}^s\fX(M) \) projecting to \( p \).
\end{lemma}

\begin{proof}
  The left-hand side is independent of the presentation of \( p \) and
  both sides are \( \bR \)-linear, so it is enough to prove the
  corresponding formula for a decomposable \( Q \) projecting to \( q
  \) as in~\eqref{eq:Qq}.  Note that when \( s = r+1 \) we have one of
  the standard formulae for the exterior derivative:
  \begin{equation}
    \label{eq:dd}
    (d\alpha)(q) = \Ldc_Q\alpha - \alpha(L(Q)).
  \end{equation}
  For general \( s \leqslant r+1 \), write \( Q' = Y_1 \curlywedge
  \dots \curlywedge Y_t \) with \( s + t = r+1 \).  Note that we
  always have \( Q \hook \beta = q \hook \beta \) for any form \(
  \beta \).  Now we compute, using \eqref{eq:dd}, \eqref{eq:Ldc} and
  \eqref{eq:L},
  \begin{equation*}
    \begin{split}
      (q &\hook d\alpha)(q') = d\alpha(q \wedge q')
      = \Ldc_{Q \curlywedge Q'}\alpha - \alpha(L(Q \curlywedge Q'))\\
      &= \sum_{i=1}^s (\Ld_{X_i}\alpha)(Q_{\wedge i} \curlywedge Q') +
      (-1)^s \sum_{j=1}^t (\Ld_{Y_j}\alpha)(Q \curlywedge Q'_{\wedge
      j}) \eqbreak - \alpha(L(Q) \curlywedge Q') - (-1)^s \alpha(Q
      \curlywedge L(Q')) \eqbreak \qquad - (-1)^s \sum_{i=1}^s
      \sum_{j=1}^t \alpha([X_i,Y_j] \curlywedge
      Q_{\wedge i} \curlywedge Q'_{\wedge j}) \\
      &= (\Ldc_Q\alpha)(Q') + (-1)^s \sum_{j=1}^t \Ld_{Y_j}(Q \hook
      \alpha)(Q'_{\wedge j}) - (-1)^s \sum_{j=1}^t \alpha(\Ld_{Y_j}Q
      \curlywedge Q'_{\wedge j}) \eqbreak - (L(Q)\hook\alpha)(Q') -
      (-1)^s(Q \hook \alpha)(L(Q')) +
      (-1)^s \sum_{j=1}^t \alpha(\Ld_{Y_j}Q \curlywedge Q'_{\wedge j}) \\
      &= \bigl(\Ldc^{}_Q\alpha + (-1)^s d(q\hook \alpha) - L(Q) \hook
      \alpha\bigr)(q'),
    \end{split}
  \end{equation*}
  which gives the claimed result.
\end{proof}

\subsection{Symmetries and multi-moment maps}
\label{sec:symmetries}

Let \( (M,\alpha) \) be a manifold with an \( r \)-form \( \alpha \)
not necessarily closed.

\begin{definition}
  A \emph{group of symmetries} of \( (M,\alpha) \) is a connected Lie
  group \( G \) acting on \( M \) preserving \( \alpha \).
\end{definition}

\noindent
Infinitesimally this means that
\begin{equation*}
  \Ld_X \alpha = 0\quad\text{for all \( \Lel X \in \g \)},
\end{equation*}
where \( \g \) is the Lie algebra of \( G \) and we write \( X \) for
the vector field generated by~\( \Lel X \).

\begin{example}[Multi-phase space]
  \label{ex:multi-phase-cont}
  Suppose \( M = \Lambda^kT^*N \) with the canonical \( k \)-plectic
  form \( \alpha \) of \cref{ex:multi-phase}.  Then any diffeomorphism
  \( \phi \) of the base~\( N \) induces a symmetry \( \psi \) of \(
  (M,\alpha) \) covering~\( \phi \), namely take \( \psi =
  (\phi^*)^{-1} \).  In this way, any group \( G \) of diffeomorphisms
  of \( N \) lifts a group of symmetries of \( (M,\alpha) \).
\end{example}

The map sending an element \( \Lel X \) of \( \g \) to the vector
field \( X \) on \( M \) generated by \( \Lel X \) is \( \bR
\)-linear.  So we may extend this to associate to each \( \Lel p \in
\Lambda^s \g \) a unique multi-vector \( p \in \fX^s(M) \).  For a
decomposable \( \Lel q = \Lel X_1 \wedge \Lel X_2 \wedge \dots \wedge
\Lel X_s \), the corresponding multi-vector is exactly the \( q \)
given in equation~\eqref{eq:qq}.  When \( G \) preserves \( \alpha \),
we have \( \Ldc_p\alpha = 0 \) for each \( \Lel p \in \Lambda^s\g \),
so the extended Cartan formula reads
\begin{equation}
  \label{eq:Cartan-invariant}
  p\hook d\alpha - (-1)^sd(p\hook \alpha) = - L(p) \hook
  \alpha\quad\text{for all \( \Lel p \in \Lambda^s\g \)},
\end{equation}
where \( L(p) \) is understood to be the multi-vector corresponding to
\( L(\Lel p) \in \Lambda^{s-1}\g \) which is defined as
in~\eqref{eq:L} but using the Lie bracket of \( \g \).  In particular,
when \( \alpha \) is closed, we see that \( p\hook \alpha \) is closed
whenever \( L(\Lel p) = 0 \).

\begin{definition}
  The \emph{\( k \)th Lie kernel} of \( \g \) is the \( \g \)-module
  \begin{equation*}
    \Pgk = \ker( L\colon \Lambda^k\g \to \Lambda^{k-1}\g ).
  \end{equation*}
  If \( G \) acts a group of symmetries for a closed geometry of
  degree \( r \) we may write
  \begin{equation*}
    \Pg = \Pg[\g,r-1]
  \end{equation*}
  for the \emph{corresponding Lie kernel} of \( \g \). 
\end{definition}

\noindent
Since each \( \Ad_g \), \( g\in G \), is a Lie algebra automorphism of
\( \g \), we see that \( \Pg \) is a \( G \)-module.

If \( G \) is a Abelian, then \( \Pgk = \Lambda^k\g \).  For any \( \g
\), we have \( \Pg[\g,1] = \g \).

\begin{example}
  \label{ex:su2}
  For \( G = \SU(2) \), let \( \Lel X_1,\Lel X_2,\Lel X_3 \) be a
  basis of \( \su(2) \) satisfying \( [\Lel X_1,\Lel X_2] = -2\Lel X_3
  \), etc.  The map \( L\colon \Lambda^3\su(2) \to \Lambda^2\su(2) \)
  is given by the cyclic sum \( L(\Lel X_1\wedge \Lel X_2\wedge \Lel
  X_3) = \Cyclic_{1,2,3}[\Lel X_1,\Lel X_2]\wedge \Lel X_3 =
  -2\sum_{i=1}^3 \Lel X_i\wedge \Lel X_i = 0 \), so this \( L \) is
  identically zero.  Thus \( \Pg[\su(2),2] = \Lambda^3\su(2)^* \cong
  \bR \).  On the other hand, \( L\colon \Lambda^2\su(2) \to \su(2) \)
  is an isomorphism and \( \Pg[\su(2),1] = \{0\} \).

  More generally, if \( \g \) is a simple Lie group of compact type,
  and \( \Lel X_1,\dots,\Lel X_n \) is a basis for \( \g \) consisting
  of unit length vectors for the Killing form, then one has that the
  element \( \sum_{i,j=1}^n \Lel X_i \wedge \Lel X_j \wedge [\Lel
  X_i,\Lel X_j] \) lies in \( \Pg[\g,2] \).  This element corresponds
  to the standard representative \( \gamma(\Lel X,\Lel Y,\Lel Z) =
  \inp{[\Lel X,\Lel Y]}{\Lel Z} \) of the third cohomology group \(
  H^3(\g) \cong \bR \).  For \( \g \) simple, \( L\colon \Lambda^2\g
  \to \g \) is onto so we have \( \Pg[\g,1] \cong {\Lambda^2\g} / \g
  \), which is a non-zero irreducible \( G \)-module when \( \dim G >
  3 \), cf. \textcite{Wolf:isotropy}.
\end{example}

Now suppose that we have a closed geometry \( (M,\alpha) \) and that
\( G \) is a group of symmetries.  Then for \( \Lel p \) in \( \Pg =
\Pg[\g,r-1] \leqslant \Lambda^{r-1}\g \) we have \( p \hook \alpha \in
\Omega^1(M) \) and
\begin{equation*}
  d(p \hook \alpha) = 0,
\end{equation*}
by \cref{lem:g-Cartan}.  Thus \( p \hook \alpha \) is a closed
one-form and locally the derivative of a function \( \nu_{\Lel p} \).
Letting \( \Lel p \) vary over \( \Pg \), we obtain a family of
functions that may be combined into a local map \( \nu \colon M \to
\Pg^* \) by setting \( \inp \nu{\Lel p} = \nu_{\Lel p} \).  This
motivates the following definition:

\begin{definition}
  \label{def:multi}
  Let \( G \) be a symmetry group for a closed geometry \( (M,\alpha)
  \).  A \emph{multi-moment map} for this action is an equivariant map
  \( \nu \colon M \to \Pg^* \) satisfying
  \begin{equation}
    \label{eq:multi}
    d \inp \nu{\Lel p} = p\hook\alpha
  \end{equation}
  for all \( \Lel p\in \Pg \).
\end{definition}

\noindent
For \( \alpha \) a symplectic form, this is exactly the usual notion
of moment map, since \( \Pg = \Pg[\g,1] = \g \).  Concrete examples of
multi-moment maps will be given in \cref{sec:example-geometries}.

The obstructions to constructing a multi-moment map are quite weak.
To start with if \( b_1(M) = 0 \) then we can always find global
functions \( \nu_{\Lel p} \) with \( d\nu_{\Lel p} = p \hook \alpha
\).  Using averaging arguments we have the following topological
existence result.

\begin{theorem}
  \label{thm:topological}
  Let \( (M,\alpha) \) be a closed geometry with \( G \) acting as a
  group of symmetries.  Suppose \( b_1(M) = 0 \).  Then there exists a
  multi-moment map \( \nu \colon M \to \Pg^* \) if either
  \begin{compactenum}
  \item \( G \) is compact, or
  \item \( M \) is compact and orientable, and \( G \) preserves a
    volume form on \( M \).
  \end{compactenum}
\end{theorem}

\begin{proof}
  The proofs given in~\cite{Madsen-S:multi-moment} for the case of \(
  \alpha \) a three-form carry directly over to this general
  situation.  In the first case, one averages over \( G \); in the
  second, one averages over~\( M \).
\end{proof}

A second useful existence result occurs when \( \alpha \) is exact in
a good way.

\begin{proposition}
  \label{prop:exact}
  Suppose \( G \) is a group of symmetries of a closed geometry \(
  (M,\alpha) \) of degree~\( r \).  If there exists a \( G
  \)-invariant form \( \beta \) with \( \alpha = d\beta \), then
  \begin{equation*}
    \inp \nu{\Lel p} = (-1)^{r-1}\beta(p)\qquad\text{for \( \Lel p \in
    \Pg \)}
  \end{equation*}
  defines a multi-moment map for the action of \( G \).
\end{proposition}

\begin{proof}
  Invariance of \( \beta \) implies that \( \nu \) is equivariant, so
  we just need to verify equation~\eqref{eq:multi}.  However,
  \cref{lem:g-Cartan} gives
  \begin{equation*}
    d \inp \nu{\Lel p} = (-1)^{r-1} d(p \hook \beta)
    = p \hook d\beta = p \hook \alpha,
  \end{equation*}
  since \( L(\Lel p) = 0 \) and \( \beta \) is invariant.
\end{proof}

Finally there is a purely algebraic existence and uniqueness result
depending only on the Betti numbers of \( \g \).  The dual of the map
\( L\colon \Lambda^{k+1}\g \to \Lambda^k\g \) is essentially the
differential
\begin{equation}
  \label{eq:d-g}
  \begin{gathered}
    d\colon \Lambda^k\g^* \to \Lambda^{k+1}\g^*\\
    (d\gamma)(\Lel X_1,\Lel X_2,\dots,\Lel X_k) = -\gamma(L(\Lel X_1
    \wedge \Lel X_2 \wedge \dots \wedge \Lel X_k)).
  \end{gathered}
\end{equation}
The Jacobi identity implies \( L\circ L = 0 \) and so \( d \circ d = 0
\).  Thus we have the Lie algebra homology \( H_*(\g) \) of \( \g \)
defined by the complex \( (\Lambda^*{\g},L) \) and the Lie algebra
cohomology \( H^*(\g) \) defined by \( (\Lambda^*{\g}^*,d) \).  In
particular,
\begin{equation}
  \label{eq:Hk-g}
  H^k(\g) = \frac{\ker(d\colon \Lambda^k\g^* \to
  \Lambda^{k+1}\g^*)}{\im(d\colon \Lambda^{k-1}\g^* \to
  \Lambda^k\g^*)} = \frac{Z^k(\g)}{B^k(\g)},
\end{equation}
the quotient of the space \( Z^k(\g) = \ker d \) of cycles by the
space of boundaries \( B^k(\g) = \im d \), and we write
\begin{equation*}
  b_k(\g) = \dim H^k(\g)
\end{equation*}
for the \emph{\( k \)th Betti number} of \( \g \).

Our algebraic existence and uniqueness criteria are expressed in terms
of vanishing of certain Betti numbers.  It is therefore useful to
introduce the following terminology.

\begin{definition}
  A connected Lie group \( G \) or its Lie algebra \( \g \) is
  \emph{(cohomologically) \( (k_1,k_2,\dots,k_\ell) \)-trivial} if the
  Betti numbers \( b_k(\g) \) vanish for \( k = k_1, k_2,\dots, k_\ell
  \).
\end{definition}

We will discuss these type of conditions in some detail in
\cref{sec:cohomology}, however let us note that a simple Lie algebra
is always \( (1,2) \)-trivial, but has \( b_3 \) non-zero.  Indeed
looking up the Poincaré polynomials of the compact simple Lie algebras
reveals the following:

\begin{proposition}[\textcite{Madsen:mm-thesis}]
  Every compact simple Lie algebra not isomorphic to \( \su(n) \), \(
  n\geqslant3 \), is \( (1,2,4,5,6) \)-trivial.  \qed
\end{proposition}

Returning to multi-moment maps the algebraic existence and uniqueness
result is:

\begin{theorem}
  \label{thm:algebraic}
  Suppose \( (M,\alpha) \) is a closed geometry of degree \( r \) and
  \( G \) is a group of symmetries.  If \( G \) is \( (r-1,r)
  \)-trivial then there exists a unique multi-moment map \( \nu\colon
  M \to \Pg^* \).  If \( G \) is just \( (r-1) \)-trivial, then \( \nu
  \) is unique whenever it exists.
\end{theorem}

\begin{proof}
  The proof builds on the following observation.  Taking the dual of
  the exact sequence
  \begin{equation*}
    \begin{CD}
      0 @>>> \Pgk @>\iota>> \Lambda^k\g @>L>> \Lambda^{k-1}\g
    \end{CD}
  \end{equation*}
  we obtain the sequence
  \begin{equation*}
    \begin{CD}
      \Lambda^{k-1}\g^* @>d>> \Lambda^k\g^* @>\pi>> \Pgk^* @>>> 0.
    \end{CD}
  \end{equation*} 
  From this one sees that \( \Pgk^* \cong \Lambda^k\g^* / B^k(\g) \)
  and so the exterior derivative \( d \colon \Lambda^k\g^* \to
  \Lambda^{k+1}\g^* \) induces a well-defined linear map \( \dP\colon
  \Pg^* \to \Lambda^{k+1}(\g) \) via \( \dP a = db \) where \( a = b +
  B^k(\g) \in \Pgk^* \cong \Lambda^k\g^* / B^k(\g) \).  We now see
  that \( \dP \) is injective if and only if \( b_k(\g) = 0 \) and
  that the image of \( \dP \) is \( B^{k+1}(\g) \).

  Now let us consider the situation of the \namecref{thm:algebraic}.
  The action of \( G \) on \( M \) defines a map
  \begin{gather}
    \Psi \colon M \to Z^r(\g), \notag \\
    \inp {\Psi(x)}{\Lel p} = (-1)^r \alpha(p)_x, \label{eq:Psi}
  \end{gather}
  for all \( \Lel p \in \Lambda^r\g \) and \( x \in M \).  To see that
  the image lies in \( Z^r(\g) \leqslant \Lambda^r\g^* \), use
  \eqref{eq:Cartan-invariant} for the invariant closed form \( \alpha
  \) to get
  \begin{equation*}
    \inp {d(\Psi(x))}{\Lel q} = \inp {\Psi(x)}{L(\Lel q)} = (-1)^r
    \alpha(L(q))_x = (-1)^{r+1} (q \hook d\alpha)_x = 0, 
  \end{equation*}
  for each \( \Lel q \in \Lambda^{r+1}\g \).

  Now if \( b_r(\g) = 0 \), then \( Z^r(\g) = B^r(\g) = \im \dP \), so
  we may find for each \( x\in M \) a \( \nu_x \in \Pg^* \) with \(
  \dP(\nu_x) = \Psi(x) \).  If \( b_{r-1}(\g) = 0 \), the map \( \dP
  \) is injective, so there is a unique choice of \( \nu_x \) for each
  \( x \).  It follows that \( \nu \) is equivariant.

  Suppose we have an equivariant map \( \nu\colon M \to \Pg^* \) with
  \( \dP\nu = \Psi \) of equation~\eqref{eq:Psi} and that \(
  b_{r-1}(\g) = 0 \).  We claim that \( \nu \) is a multi-moment map.
  The important fact here is that \( b_{r-1}(\g) = 0 \) says \( \ker d
  = \im d \) in \( \Lambda^{r-1}\g^* \) which dually means that \( \im
  L = \ker L \) in \( \Lambda^{r-1}\g \).  However, \( \ker L = \Pg \)
  so \( L\colon \Lambda^r\g \to \Lambda^{r-1}\g \) maps on to the Lie
  kernel~\( \Pg \).  We may now compute, for \( \Lel p = -L(\Lel q) \in
  \Pg \),
  \begin{equation*}
    \begin{split}
      d \inp \nu{\Lel p} &= -d \inp \nu{L(\Lel q)}
      = d \inp{\dP(\nu)}{\Lel q} \\
      &= d \inp \Psi{\Lel q} = (-1)^r d(q\hook \alpha) \\
      &= L(q) \hook \alpha = p \hook \alpha,
    \end{split}
  \end{equation*}
  by \eqref{eq:Cartan-invariant}.  Thus \( \nu \) is indeed a
  multi-moment map.
\end{proof}

\section{Example geometries and their multi-moment maps}
\label{sec:example-geometries}

Having introduced the general theory of multi-moment maps we will now
look at a number of concrete examples.  For many examples the main
focus will be on closed geometries of degree~\( 4 \), but we will also
consider other cases.  When relevant we will also discuss the use of
multi-moment maps to describe reductions of certain geometries.

\begin{definition}
  Suppose \( \nu \colon M \to \Pg^* \) is a multi-moment map.  Then
  for each \( t \in \Pg^* \) fixed by the \( G \)-action, the
  \emph{reduction of \( M \) at level~\( t \)} is
  \begin{equation*}
    M \mred[\nu,t] G = \nu^{-1}(t) / G.
  \end{equation*}
  We set \( M \mred G = M \mred[\nu,0] G \) for the reduction at
  level~\( 0 \).
\end{definition}

This makes sense, since \( \nu \) is \( G \)-equivariant so the \( G
\)-action preserves \( \nu^{-1}(t) \) whenever \( t \) is fixed by~\(
G \).  This notion of reduction corresponds to the usual
Marsden-Weinstein quotient in symplectic geometry.  However because
the structure of forms of higher degree is so varied, the type of
geometry obtained on the quotient is often of a different character
to the geometry on~\( M \).  Also, one usually has to impose
assumptions, such as freeness of the action of \( G \) and regularity
of the value~\( t \), in order to obtain smooth quotients.

\subsection{Multi-phase space}
\label{sec:multi-phase-space}

This is \( M = \Lambda^{r-1}T^*N \) with the canonical \( r \)-form \(
\alpha \) of \cref{ex:multi-phase}.  If \( G \) is any group of
diffeomorphisms of \( N \), then as noted in
\cref{ex:multi-phase-cont}, this induces an action of \( G \) on \( M
\) preserving \( \alpha \).  However, in this case we have \( \alpha
\) is equivariantly exact: the canonical form \( \beta \) is also \( G
\)-invariant and satisfies \( d\beta = \alpha \).  By
\cref{prop:exact}, there is a multi-moment map \( \nu \) given by \(
\inp \nu{\Lel p} = (-1)^{r-1}\beta(p) \).

A concrete example is provided by taking \( N = \bR^4 \).  If we
consider the closed geometry of degree \( 4 \) on \( M = \Lambda^3T^*N
\cong TN \cong \bR^8 \), we have \( \alpha \) and \( \beta \) given by
the cyclic sums
\begin{equation*}
  \alpha = \Cyclic_{1,2,3,4} dp^1 \wedge dq^2 \wedge dq^3 \wedge dq^4,\quad
  \beta = \Cyclic_{1,2,3,4} p^1 dq^2 \wedge dq^3 \wedge dq^4.
\end{equation*}
If \( G = \bR^4 \) acts by translations on \( N = \bR^4 \) then
\begin{equation*}
  \nu\left(\frac\partial{\partial q^1} \wedge \frac\partial{\partial
    q^2} \wedge \frac\partial{\partial q^3}\right) =
  p^4,\quad\text{etc.} 
\end{equation*}
and \( \nu \) is simply projection on to the fibres of \( T\bR^4 \to
\bR^4 \).

\subsection{Product manifolds}
\label{sec:product-manifolds}

Let \( (N,\alpha') \) be a \( (k-1) \)-plectic manifold.  Consider \(
M = S^1 \times N \) and write \( \theta \) for the standard one-form
on the \( S^1 \)-factor.  Then as in the proof of
\cref{prop:non-degenerate}, we have that \( \alpha = \theta \wedge
\alpha' \) is a \( k \)-plectic form on~\( M \).

If \( H \) is a group of symmetries of \( (N,\alpha') \), then \( G =
S^1 \times H \) is a group of symmetries of \( (M,\alpha) \), where
the \( S^1 \)-factor of~\( G \) acts non-trivially on just the \( S^1
\)-factor of \( M \) preserving~\( \theta \).

Suppose \( \nu'\colon N \to \Pg[\h,k-1]^* \) is a multi-moment map for
the action of~\( H \) on~\( N \).  Writing \( \g = \bR\Lel T \oplus \h
\), we have that
\begin{equation}
  \label{eq:Lkg}
  \Lambda^m \g = (\Lel T \wedge \Lambda^{m-1}\h) \oplus \Lambda^m \h.
\end{equation}
Since \( \Lel T \) commutes with \( \h \), the map \( L \colon
\Lambda^k \g \to \Lambda^{k-1} \g \) preserves the
splittings~\eqref{eq:Lkg} and we conclude that
\begin{equation*}
  \Pg[\g,k] \cong (\Lel T \wedge \Pg[\h,k-1]) \oplus \Pg[\h,k].
\end{equation*}
As \( 0 = \Ld_T\theta = d(T\hook\theta) \), we may scale \( \Lel T
\) by a constant so that \( \theta(T) = 1 \).  Let~\( \vartheta \)
denote the element of~\( \g^* \) that annihilates~\( \h \) and has
\( \vartheta(\Lel T) = 1 \).  We claim that
\begin{equation*}
  \nu = \vartheta \wedge \nu'
\end{equation*}
is a multi-moment map for the action of \( G \) on~\( M \).  Firstly,
\( \nu \) is a map to \( \bR\vartheta \wedge \Pg[\h,k-1]^* \subset
\Pg[\g,k]^* \) and it is equivariant for the action of \( G = S^1 \times
H \).  Secondly, for \( \Lel p \in \Pg[\g,k] \), we have \( \Lel p =
\Lel T \wedge \Lel p' + \Lel q \) with \( \Lel p' \in \Pg[\h,k-1] \)
and \( \Lel q \in \Pg[\h,k] \).  Now \( \nu \) is zero on \( \Lel q \) and
\begin{equation*}
  d\inp \nu{\Lel p} = d\inp{\vartheta\wedge\nu'}{\Lel T\wedge\Lel p'}
  = d\inp{\nu'}{\Lel p'} = p' \hook \alpha' = (T\wedge p') \hook
  (\theta\wedge\alpha') = p \hook \alpha.
\end{equation*}
So \( \nu \) satisfies \eqref{eq:multi} and is a multi-moment map.

Starting with \( (N,\omega) \) a symplectic manifold with a
Hamiltonian action of~\( H \), iteration of the above construction
produces a \( k \)-plectic structure on \( M = T^{k-1} \times N \)
together with a multi-moment map for the action of \( G = T^{k-1}
\times H \).

\subsection{Symplectic manifolds}
\label{sec:sympelctic}

If \( \omega \in \Omega^2(M) \) is an ordinary symplectic form, then
each power \( \omega^k \in \Omega^{2k}(M) \) with \( 2k \leqslant \dim
M \) is \( (2k-1) \)-plectic.  In particular, we may consider the
four-form \( \alpha = \omega \wedge \omega = \omega^2 \) as a \( 3
\)-plectic form on \( M \).  Let us take \( \dim M \geqslant 6 \) and
assume that \( M \) is simply-connected.

If \( X \) is a vector field preserving \( \alpha \), then we have \(
0 = \Ld_X\alpha = 2\omega\wedge\Ld_X\omega \).  But the map \(
\omega\wedge\cdot \colon \Lambda^2T^*M \to \Lambda^4T^*M \) is
injective when \( \dim M \geqslant 6 \), so we have \( \Ld_X\omega = 0
\) and \( X \)~also preserves \( \omega \).  Thus symmetries of \(
(M,\omega^2) \) are nothing but symplectomorphisms of \( (M,\omega) \).

Let us first consider actions of Abelian groups.  Suppose \( G = \bR^3
\) acts generated by vector fields \( X_1 \), \( X_2 \) and~\( X_3 \).
Then by the extended Cartan formula (\cref{lem:g-Cartan}), we have \(
d(\omega(X_i,X_j)) = d((X_i \wedge X_j) \hook \omega) = -[X_i,X_j]
\hook \omega = 0 \), showing that \( \omega(X_i,X_j) \) is constant.
Taking constant linear combinations of our vector fields we may
therefore assume that \( \omega(X_i,X_3) = 0 \), for \( i=1,2 \), and
that \( \omega(X_1,X_2) = \delta \in \{0,1\} \).

Now a multi-moment map \( \nu \colon M \to \Pg^* = \Lambda^3\g^* \cong
\bR \) has differential
\begin{equation}
  \label{eq:omega2}
  \begin{split}
    d\nu &= (X_1\wedge X_2\wedge X_3) \hook \alpha = (X_1\wedge
    X_2\wedge X_3) \hook \omega^2 \\
    &= \Cyclic_{1,2,3} \omega(X_1,X_2)\omega(X_3,\cdot)
    = \delta\, X_3\hook\omega.
  \end{split}
\end{equation}
Thus if \( \delta = 0 \), i.e. the orbits of the \( G \)-action are
isotropic, then \( \nu \) is constant.  Fixing now \( \delta = 1 \),
we note that if \( \mu_i \) is a symplectic moment map for \( X_i \),
then \( X_2\mu_1 = d\mu_1(X_2) = \omega(X_1,X_2) = \delta = 1 \).
Thus in this case \( X_1 \) and \( X_2 \) do not Poisson commute and
there is no symplectic moment map for the action of the whole of~\( G
\).  However, \( \nu = \mu_3 \) is a multi-moment map for the action
of~\( G \).

A second case is given by considering \( G = \SU(2) \) with generators
\( \Lel X_i \in \su(2) \) as in \cref{ex:su2}.  Since \( G \) is
compact and \( b_1(M) = 0 \), there is a symplectic moment map \( \mu
= (\mu_1,\mu_2,\mu_3) \colon M \to \su(2)^*\cong \bR^3 \) satisfying
\( L_{X_1}\mu_2 = -2\mu_3 \), etc.  As above we find that \(
L_{X_1}\mu_2 = d\mu_2(X_1) = -\omega(X_1,X_2) \).  For a multi-moment
map \( \nu \colon M \to \Pg[\su(2),2] \cong \bR \),
equation~\eqref{eq:omega2} gives
\begin{equation*}
  d\nu =  (X_1\wedge X_2\wedge X_3) \hook \alpha =  2
  \sum_{i=1}^3 \mu_i d\mu_i = d\norm{\mu}^2.
\end{equation*}
Thus a multi-moment map is \( \nu = \norm{\mu}^2 \).  It is unique up
to the addition of a constant. 

The quotient \( M \mred \SU(2) = \nu^{-1}(0)/\SU(2) \) is nothing
other than the symplectic quotient of \( M \) by \( \SU(2) \), and
thus inherits both a symplectic form \( \omega' \) and a closed
four-form \( \alpha' = (\omega')^2 \).

For \( t>0 \), the geometry of the reduction \( M \mred[\nu,t] \SU(2)
= \nu^{-1}(t)/\SU(2) \) is more complicated.  Part of the reason for
this is that even in the good case when \( \SU(2) \) acts freely on \(
M_t \), there is no canonical choice of connection form for the \(
\SU(2) \)-bundle \( M_t \to M \mred[\nu,t] \SU(2) \).  This problem is
remedied in geometries that come equipped with a metric.

\subsection{HyperKähler manifolds}
\label{sec:hK}

A variant of the construction considered above arises in the setting
of quaternionic geometry. A quaternion-Hermitian manifold \( Q \) is a
\( 4n \)-dimensional Riemannian manifold with a rank three subbundle
\( \mathcal G \subset \End(TQ) \) which is locally trivialised by
anti-commuting almost complex structures \( I \), \( J \) and \( K \)
that satisfy \( K = IJ \). In addition the Riemannian metric \( g \)
must be compatible with \( \mathcal G \), meaning \( g(\mathcal I
X,\mathcal I Y) \) for each \( X,Y\in T_xQ \) and \( \mathcal I\in
\mathcal G_x \); in particular \( \omega_I = g(I\cdot,\cdot) \), etc.,
are locally defined non-degenerate two-forms. A quaternion-Hermitian
manifold carries a non-degenerate four-form \( \Omega \) which may
locally be expressed as
\begin{equation*}
  \Omega = \omega_I\wedge\omega_I + \omega_J\wedge\omega_J +
  \omega_K\wedge\omega_K. 
\end{equation*}

In dimension eight and above one says that \( Q \) is
\emph{quaternionic Kähler} if the fundamental form is parallel, \(
\LC\Omega = 0 \).  This implies that \( \Omega \) is closed.  In
dimensions \( 12 \) and higher \( d\Omega = 0 \) is actually
equivalent to the quaternionic Kähler condition
\cite{Swann:symplectiques}.  \Textcite{Gray:Sp} showed that the
stabiliser of \( \Omega \) under the action of \( \GL(4n,\bR) \), \(
n>1 \), is the compact group \( \Sp(n)\Sp(1) \) and so \( \Omega \)
determines the metric \( g \).  In dimension four, these
considerations no longer hold and a quaternionic Kähler manifold is
instead defined to be an oriented Riemannian manifold which is
Einstein and self-dual.

If the subbundle \( \mathcal G \) can be globally trivialised by \(
I,J,K \) and these almost complex structures are integrable then we
have a hyper-Hermitian manifold. This will then be hyperKähler
provided that the two-forms \( \omega_I \), etc., are closed.

In \cite{Swann:MathAnn} it was shown that to any quaternionic Kähler
manifold \( Q^{4n} \) of positive scalar curvature one may associate a
hyperKähler manifold \( M^{4n+4} = \mathcal{U}(Q) \) which acts as a
hyperKähler generalisation of the twistor space; this is known as the
\emph{Swann bundle} and may be written as \( \mathcal U(Q) = \bR_{>0}
\times \mathcal S \), where \( \mathcal S \) is the bundle of triples
\( (I,J,K) \). Conversely given a \( (4n + 4) \)-dimensional
hyperKähler manifold \( M \) admitting a special type of \( \SU(2)
\)-action then a version of the Marsden-Weinstein reduction produces a
quaternionic-Kähler manifold of positive scalar curvature; this latter
reduction process can be realised in terms of multi-moment maps.

The relevant type of \( \SU(2) \)-symmetry often arises due to the
presence of a vector field \( X \) on \( (M^{4n+4},g,I,J,K) \), a
\emph{special homothety} (cf.~\cite{Poon-S:HKT-QKT}), with the
following properties:
\begin{equation*}
  \begin{gathered}
    \Ld_Xg = g,\quad \Ld_{IX}g = 0,\quad \Ld_{IX}I = 0,\quad\Ld_{IX}J
    = -K,\quad \Ld_{IX}K = J,\quad\text{etc.}
  \end{gathered}
\end{equation*}

Special homotheties generate a local action of \( \bH^* \) and in good
cases the vector fields \( IX \), \(JX \), \( KX \) integrate to give
an action of \( \SU(2) \) which is necessarily locally free. 

\begin{proposition}[\textcite{Madsen:mm-thesis}]
  \label{thm:qKmmmap}
  Let \( (M^{4n+4},g,I,J,K) \) be a hyperKähler manifold, and \( X
  \) a special homothety.  If \( IX,JX,KX \) generate a locally free
  action of \( \SU(2) \) then this action preserves \( \Omega \) and a
  multi-moment map \( \nu\colon M\to\bR\cong\Pg[\su(2),2]^* \) is
  given by
  \begin{equation*}
    \nu = -3\norm{X}^4.
  \end{equation*}
  \qed 
\end{proposition}

Any non-zero \( t\in\nu(M) \) is a regular value.  The level sets
correspond to \( \norm{X} \) is constant and the results of
\cite{Swann:MathAnn} show that \( M \mred[\nu,t] \SU(2) \) is a
quaternionic Kähler orbifold of positive scalar curvature.

\subsection{Holonomy Spin(7)}
\label{sec:holonomy-spin7}

A \( \Spin(7) \) structure is a geometry modelled on the form \(
\Phi_0 \) of equation~\eqref{eq:Spin7}.

\begin{definition}
  An eight-manifold \( M \) has a \emph{\( \Spin(7) \)-structure} if
  there is a form \( \Phi \in \Omega^4(M) \) such that \(
  (T_xM,\Phi_x) \) is linearly isomorphic to \( (\bR^8,\Phi_0) \) for
  each~\( x \in M \).
\end{definition}

Now \( \Phi \) determines a volume form and Riemannian metric on \( M
\) via the relations
\begin{equation*}
  \Phi^2 = 14\vol,\qquad ((X\wedge Y)\hook \Phi)^2 \wedge \Phi = 6
  \norm{X \wedge Y}_g^2\vol.
\end{equation*}
Comparing with \( \Phi_0 \) we see that \( \vol \) and \( g \)
correspond to the standard volume \( \vol_0 = e_{12345678} \) and
metric \( g_0 = \sum_{i=1}^8 e_i^2 \) on \( \bR^8 \).  We also see
that \( \Phi \) is a self-dual four-form \( \Hodge \Phi = \Phi \).  In
particular, a closed \( \Spin(7) \)-structure has \( \Phi \)~harmonic
and it follows that the holonomy group of \( g \) is contained in~\(
\Spin(7) \).  This is one of the two exceptional holonomies in the
Berger classification \cite{Berger:hol,Besse:Einstein}.  Examples of
metrics with holonomy exactly \( \Spin(7) \) are not easy to find.
Local existence was proved by \textcite{Bryant:exceptional}, the first
complete examples were produced by
\textcite{Bryant-Salamon:exceptional} and the first compact examples
were found by \textcite{Joyce:Spin7}.  The complete examples produced
by Bryant \& Salamon have many symmetries, in fact the symmetry group
acts with cohomogeneity one, so the principal orbit is of codimension
one.  Further systematic study of cohomogeneity one examples with
compact symmetry group has been made by
\textcite{Reidegeld:G2,Reidegeld:Spin7-Q-M}.  One sees that many of
the examples and candidates have compact symmetry groups of rank~\( 3
\), so an interesting class of \( \Spin(7) \)-manifolds are those with
\( T^3 \)-symmetry.

Given a closed \( \Spin(7) \)-structure~\( (M,\Phi) \) with free \( T^3
\)-symmetry, fix a basis \( \Lel U_1,\Lel U_2,\Lel
U_3 \) for \( \lt \cong \bR^3 \).  Then we have the following
two-forms on \( M \):
\begin{equation}
  \label{eq:omega-i}
  \omega_1 = U_2 \hook U_3 \hook \Phi,\quad
  \omega_2 = U_3 \hook U_1 \hook \Phi,\quad
  \omega_3 = U_1 \hook U_2 \hook \Phi.
\end{equation}
These forms are all closed, by the \cref{lem:g-Cartan}. 

To see the structure of these forms, we consider the geometry of \(
(\bR^7,\Phi_0) \).  Isolating \( e_1 \) in the
expression~\eqref{eq:Spin7} for \( \Phi_0 \) we have
\begin{equation}
  \label{eq:Spin7-G2}
  \begin{split}
    \Phi_0 &= e_1\wedge (e_{234} + e_{256} + e_{278} + e_{357} -
    e_{368} - e_{458} - e_{467}) \eqbreak + e_{5678} + e_{3478} +
    e_{3456} + e_{2468} - e_{2457} - e_{2367} - e_{2358} \\
    &= e_1\wedge \varphi_0 + \Hodge_7\varphi_0,
  \end{split}
\end{equation}
where we recognise \( \varphi_0 \) on \( V_7 = \Span{E_2,\dots,E_8} \)
as a rewritten version of \( \phi_0 \) in~\eqref{eq:G2} and \(
\Hodge_7 \) is the Hodge star operator on \( V_7 \) with respect to
the induced metric and volume.  In particular, we see that the
stabiliser of \( e_1 \) under the action of \( \Spin(7) \) is the
stabiliser of \( \varphi_0 \) which is the exceptional group~\( G_2
\).  The orbit of \( e_1 \) under the action of the compact group \(
\Spin(7) \) is thus of dimension \( \dim\Spin(7) - \dim G_2 = 21 - 14
= 7 \).  As the \( \Spin(7) \)-action preserves the metric \( g_0 \),
we conclude that \( \Spin(7) \) acts transitively on the unit sphere
\( S^7 \subset \bR^8 \).  Thus for any unit vector \( v \in \bR^8 \),
we have that \( v\hook \Phi_0 \) is a \( G_2 \)-form on \( v^\bot \).

We may now repeat this argument, isolating \( e_2 \) in the
expression for~\( \varphi_0 \) to get
\begin{equation*}
  \begin{split}
    \varphi_0 &= e_2\wedge(e_{34} + e_{56} + e_{78}) \eqbreak
    + e_{357} - e_{368} - e_{458} - e_{468} \\
    &= e_2 \wedge \omega + \psi_+.
  \end{split}
\end{equation*}
This time \( \omega \) is a symplectic form on \( V_6 =
\Span{E_3,\dots,E_8} \) and \( \psi_+ \) is the real part of the
complex volume form \( (e_3 + ie_4)\wedge(e_5+ie_6)\wedge(e_7+ie_8) \)
on \( V_6 \cong \bR^6 \cong \bC^3 \).  The stabiliser of \( e_2 \)
under the action of \( G_2 \) is the stabiliser of the pair \(
(\omega,\psi_+) \) which is~\( \SU(3) \).  The orbit of \( e_2 \in V_7
\) has dimension \( \dim G_2 - \dim \SU(3) = 14 - 8 = 6 \) and is just
the unit sphere \( S^6 \subset V_7 \).  Finally, the orbit of \( e_3
\in V_6 \) under the action of \( \SU(3) \) is \( S^5 = \SU(3) /
\SU(2) \).  This demonstrates the following well-known result,
cf. \textcite{Bryant:exceptional}.

\begin{lemma}
  \( \Spin(7) \) acts transitively on orthonormal pairs \( (v_1,v_2)
  \) and orthonormal triples \( (v_1,v_2,v_3) \) in~\( \bR^8 \).  \qed
\end{lemma}

In particular, \( \Spin(7) \)~acts transitively on the sets of unit
length simple bivectors \( u_1 \wedge u_2 \) and unit length simple
trivectors \( u_1 \wedge u_2 \wedge u_3 \) on~\( \bR^8 \).  Using the
first of these statements, we describe the form \( \omega_3 \)
of~\eqref{eq:omega-i} as pointwise corresponding to a multiple of \(
(E_1\wedge E_2)\hook \Phi_0 = e_{34} + e_{56} + e_{78} \).  Thus each
\( \omega_i \) is a two-form of rank~\( 6 \).

Furthermore, suppose that \( \nu \colon M \to \bR \) is a multi-moment
map for the action of \( T^3 \).  Then \( d\nu = (U_1\wedge U_2\wedge
U_3) \hook \Phi \) which corresponds to a multiple of \( (E_1\wedge
E_2\wedge E_3) \hook \Phi_0 = e_4 \).  This implies that on a level
set \( \nu^{-1}(t) \subset M \), the pull-back \( i^*\omega_3 \) of \(
\omega_3 \) under the inclusion map~\( i \) corresponds to a multiple
of \( e_{56} + e_{78} \).  One may prove that the \( T^3 \)-invariant
forms \( i^*\omega_1 \), \( i^*\omega_2 \), \( i^*\omega_3 \) each
vanishes on \( U_1 \), \( U_2 \) and \( U_3 \) and thus they descend
to two-forms on the four-manifold \( N_t = M \mred[\nu,t] T^3 \).  The
following terminology will be used:

\begin{definition}
  A triple of \( \sigma_1,\sigma_2,\sigma_3 \) of symplectic
  structures on a manifold \( N \) of dimension four is
  \emph{weakly coherent} if the forms are pointwise linearly
  independent, define the same orientation and the pairing \( \sigma_i
  \wedge \sigma_j \) has definite sign.
\end{definition}

A more detailed analysis of our situation gives the following
description of the quotients \( N_t \).

\begin{proposition}[\textcite{Madsen:Spin7T3}]
  Let \( (M,\Phi) \) be a closed \( \Spin(7) \)-structure.  Suppose \(
  T^3 \) acts freely on \( M \) preserving \( \Phi \) and with a
  multi-moment map \( \nu \).  Then for each \( t \in \nu(M) \), the
  four-manifold \( N_t = M \mred[\nu,t] T^3 \) admits a real-analytic
  weakly coherent triple of symplectic structures \(
  \sigma_1,\sigma_2,\sigma_3 \).  \qed
\end{proposition}

\noindent
The fact that the quotient geometry is real-analytic follows from the
remark that \( T^3 \) acts by isometries and that closed \( \Spin(7)
\) structures are Ricci-flat.  It follows that the vector fields \(
U_i \) are real-analytic.

To fully describe the relationship between the geometries of \( M \)
and~\( N_t \), one may construct a connection one-form \( \theta \in
\Omega^1(M,\lt) \) as follows. Let \( G = (g_{ij}) \) with \( g_{ij} =
g(U_i,U_j) \) and let \( U^\flat = (U_1^\flat,U_2^\flat,U_3^\flat) \),
where \( U_i^\flat = g(U_i,\cdot) \).  Then \( \theta \) is given by
\begin{equation*}
  \theta = U^\flat G^{-1}.
\end{equation*}
This satisfies \( \theta_i(U_j) = \delta_{ij} \) as required.  The
matrix \( G \) here turns out to be determined the geometry on \( N_t \).

\begin{lemma}[\textcite{Madsen:Spin7T3}]
  \( G^{-1} = h^2Q \), where \( h = 1/\sqrt{\det G} \) and \( Q =
  (q_{ij}) \) is the matrix given by
  \begin{equation*}
    \sigma_i \wedge \sigma_j = 2q_{ij}\vol
  \end{equation*}
  on \( N_t \).  \qed
\end{lemma}

\noindent
Here \( \vol \) denotes the volume form on \( N_t \) induced from the quotient.

As \( \nu^{-1}(t) \) is a \( T^3 \)-bundle over \( N_t \) its
curvature \( F = d\theta \) is a closed form with integral periods.
We write \( F \in \Omega^2_{\bZ}(N,\lt) \) to indicate this.
Conversely, given such an~\( F \), one may construct a principal \(
T^3 \)-bundle over \( N \) with curvature~\( F \).    Madsen
shows that \( F \) satisfies the following symmetry condition
\begin{equation}
  \label{eq:QA}
  F_i\wedge \sigma_j = F_j \wedge \sigma_i,\qquad\text{for all \( i,j \)}.
\end{equation}
It turns out that this data is sufficient to invert the construction.

\begin{theorem}[\textcite{Madsen:Spin7T3}]
  \label{thm:Spin7-T3}
  Suppose \( N \) is a connected four-manifold with a real-analytic
  weakly coherent symplectic triple \( (\sigma_1,\sigma_2,\sigma_3) \)
  and volume form \( \vol \), with the same orientation as \(
  \sigma_i^2 \).  For each real-analytic \( F \in
  \Omega^2_{\bZ}(N,\lt) \) satisfying~\eqref{eq:QA} there is a unique
  maximal connected closed \( \Spin(7) \)-structure \( (M,\Phi) \)
  with \( T^3 \)-symmetry and multi-moment map~\( \nu \), such that \(
  (N,\sigma_i,\vol,F) \) is the reduction of~\( M \) at level \( 0 \).
  \qed
\end{theorem}

The idea of the proof is to construct a \( G_2 \)-geometry on the \(
T^3 \)-bundle \( P \to N \) determined by \( F \).  One then uses a
modified variant of the Hitchin flow to extend this to a closed \(
\Spin(7) \)-geometry on a maximal open subset of \( P \times \bR \).
Explicit examples of this construction are given
in~\cite{Madsen:Spin7T3}.  Note that even the reduction of \( \bR^8 \)
by the maximal torus of \( \Spin(7) \) gives non-trivial tri-symplectic
geometries on \( \bR^4 \) that are not hyperKähler.

\subsection{G\textsubscript{2}-manifolds}
\label{sec:g2}

A \( G_2 \)-structure on a seven-manifold \( M \) is a choice of
three-form \( \phi \in \Omega^3(M) \) which on each tangent space \(
T_xM \) is linearly equivalent to the form \( \phi_0 \)
of~\eqref{eq:G2} on~\( \bR^7 \).  As in the \( \Spin(7) \) case, the
form \( \phi \) determines a volume form and metric on \( M \) via the
relation
\begin{equation}
  \label{eq:G2-g-vol}
  (X\hook\phi) \wedge (Y\hook\phi) \wedge \phi = 6g(X,Y)\vol.
\end{equation}
To interpret this formula, note that \( g \) is required to be
positive definite, which determines the sign of \( \vol \) and the
conformal class of \( g \).  Now the fact that \( \vol \) is required
to be of unit length with respect to \( g \), fixes the conformal
factor: scaling \( g \) by \( f^2 > 0 \) scales \( \vol \) by \( f^7
\) and the right-hand side of~\eqref{eq:G2-g-vol} scales by~\( f^9 \)
(cf.~\cite{Bryant:exceptional,Karigiannis:deformations}).

The form \( \phi_0 \) is stable in the sense of \cref{def:stable} and
so is its Hodge dual, the four-form \( \Hodge_7\phi_0 \).  However, the
stabiliser of \( \Hodge_7\phi_0 \) is \( G_2 \times \bZ_2 \) rather
than just~\( G_2 \), as may be seen by noting that \( -\phi_0 \)
defines the opposite orientation on~\( \bR^7 \).

There are a number of classes of \( G_2 \)-structures with closed
four-form that have particular interest for us.

\begin{definition}
  \label{def:G2-types}
  Let \( (M,\phi) \) be a \( G_2 \)-structure.  The structure is
  \begin{compactenum}
  \item \emph{cosymplectic} if \( d\Hodge_7\phi = 0 \),
  \item \emph{parallel} if \( d\phi = 0 \) and \( d\Hodge_7\phi = 0
    \),
  \item \emph{nearly parallel} if \( d\phi = 4\Hodge_7\phi \).
  \end{compactenum}
\end{definition}

Any oriented hypersurface~\( M \) of a closed \( \Spin(7) \)-manifold
\( Y \) carries a cosymplectic \( G_2 \)-structure.  Indeed write \(
i\colon M\to Y \) for the inclusion and let \( \mathbf N \) be a unit
normal to the hypersurface, then formula~\eqref{eq:Spin7-G2} gives
\begin{equation*}
  \Phi|_Y = \mathbf N^\flat \wedge \phi + \Hodge_7\phi
\end{equation*}
with \( \phi = i^*(\mathbf N \hook \Phi) \) defining the \( G_2
\)-structure.  This gives \( d\Hodge_7\phi = di^*\Phi = i^*d\Phi = 0
\), showing that the \( G_2 \)-structure is cosymplectic.

Parallel \( G_2 \)-structures are so-called because the equations \(
d\phi = 0 \) and \( d\Hodge_7\phi = 0 \) imply that \( \phi \) is
parallel for the Levi-Civita connection of \( g \), as shown by
\textcite{Fernandez-G:G2}.  This implies that \( g \) is Ricci-flat,
cf. \textcite{Bonan:G2-Spin7}, and that the holonomy is contained in \(
G_2 \).

Given a parallel \( G_2 \)-structure \( (M,\phi) \) we may consider \(
Y = S^1 \times M \) with the four-form \( \Phi = \theta \wedge \phi +
\Hodge_7\phi \), and see that \( \Phi \) gives a closed \( \Spin(7)
\)-structure.  If \( T^2 \) acts on~\( M \) preserving~\( \phi \),
then it also preserves \( \Hodge_7\phi \).  Now much as in
\cref{sec:product-manifolds}, a multi-moment map \( \nu \) for \( T^2
\) on \( (M,\phi) \) gives a multi-moment map \( \tilde\nu \) for \(
T^3 = S^1 \times T^2 \) acting on \( (Y = S^1 \times M,\Phi) \).  

The theory of reductions of \( T^3 \)-invariant closed \( \Spin(7)
\)-structures, discussed in the previous section, may now be applied
to \( T^2 \)-invariant parallel \( G_2 \)-structures.  This gives that
\( N_t = M \mred[\nu,t] T^2 \) carries \emph{coherent} triple of
symplectic structures \( \sigma_0,\sigma_1,\sigma_2 \), meaning that
they are weakly coherent and \( \sigma_0 \wedge \sigma_i = 0 \) for \(
i=1,2 \).  Also the curvature associated to the \( S^1 \)-factor of \(
Y \) is trivial, so \( \nu^{-1}(t) \to N_t \) is a \( T^2 \)-bundle
with curvature form \( F \in \Omega^2_{\bZ}(N_t,\lt) \), whose
self-dual part has no \( \sigma_0 \)-component, and which satisfies
the condition~\eqref{eq:QA}.  A direct description of this situation
is given in~\cite{Madsen-S:multi-moment}.

The third class of \cref{def:G2-types} is the manifolds of nearly
parallel \( G_2 \)-structures.  These also go under the name of weak
holonomy~\( G_2 \) in the terminology of \textcite{Gray:weak}, who
showed that the associated metric \( g \) is Einstein with positive
scalar curvature.  The simplest example of such a \( G_2 \)-structure
is the unit sphere \( S^7 \subset \bR^8 \), with the geometry induced
from the flat \( \Spin(7) \)-structure.  The discussion above shows
that \( S^7 = \Spin(7)/G_2 \) so the maximal torus \( T^3 \)~acts
preserving this geometry.  The \( \Spin(7) \)-geometry on \( \bR^8 \)
can be recovered as a warped product.

Indeed, suppose \( (M,\phi) \) is a nearly parallel \( G_2
\)-structure.  Put \( C(M) = \bR_{>0} \times M \) with the form
\begin{equation*}
  \Phi_C = s^3ds \wedge \phi + s^4\Hodge_7\phi,
\end{equation*}
where \( s \) is the parameter on \( \bR_{>0} \).  At level \( s \),
the induced structure on \( \{s\} \times M \) is given by \( s^3\phi
\), which is a \( G_2 \)-structure with metric~\( s^2g \).  This shows
that at each point \( \Phi_C \) is linearly isomorphic to the \(
\Spin(7) \)-form \( \Phi_0 \) of~\eqref{eq:Spin7}; so \( \Phi_C \)
defines a \( \Spin(7) \)-structure on~\( C(M) \).  As \( d\phi =
4\Hodge_7\phi \), the four-form \( \Hodge_7\phi \) is closed and we
find that
\begin{equation*}
  d\Phi_C = -s^3ds \wedge d\phi + 4s^3ds\wedge\Hodge_7\phi +
  s^4d\Hodge_7\phi = 0.
\end{equation*}
Thus \( (C(M),\Phi_C) \) is a closed \( \Spin(7) \)-structure.  Its
metric is the warped product \( g_C = ds^2 + s^2g \).  This
construction was used by \textcite{Baer:real} to relate the Killing
spinors of~\( (M,g) \) to parallel spinors of~\( (C(M),g_C) \).  For
more on the Killing spinor approach to these \( G_2 \)-structures see
\cite{Baum-FGK:spinors,Friedrich-KMS:G2}.

Now any symmetry of \( (M,\phi) \) induces a symmetry of \(
(C(M),\Phi_C) \) that preserves~\( s \).  Thus a nearly parallel \(
G_2 \)-structure with \( T^3 \)-symmetry corresponds to a certain
class of \( T^3 \)-invariant closed \( \Spin(7) \)-structures.  As in
the previous section, let \( U_1,U_2,U_3 \) be vector fields
generating the \( T^3 \)-action.  For \( u = U_1 \wedge U_2 \wedge U_3
\), the nearly parallel condition and the extended Cartan formula
(\cref{lem:g-Cartan}) give
\begin{equation*}
  u \hook \Hodge_7\phi = \tfrac14 u\hook d\phi = - \tfrac14 d(\phi(u)).
\end{equation*}
Thus \( \nu = -\tfrac14\phi(u) \) is a multi-moment map for the \( T^3
\)-action on \( (M,\Hodge\phi) \).  Also we have that
\begin{equation*}
  \begin{split}
    u \hook \Phi_C
    &= -s^3ds\, \phi(u) + s^4\, u \hook \Hodge_7\phi \\
    &= -\tfrac14 (d(s^4\phi(u))).
  \end{split}
\end{equation*}
So \( \nu_C = s^4\nu \) is a multi-moment map for the action on the
cone \( (C(M),\Phi_C) \).

Now for general \( t \), the level set \( \nu_C^{-1}(t) \) consists of
the \( (s,m) \) such that \( s^4\nu(m) = t \).  Since \( s\in\bR_{>0}
\), this relation simplifies when \( t=0 \) and we have that
\begin{equation*}
  \nu_C^{-1}(0) = C(\nu^{-1}(0)).
\end{equation*}
We will therefore only consider the reductions at level~\( 0 \).

The reduction \( N_C = \nu_C^{-1}(0) / T^3 \) of the cone carries a
weakly coherent triple of symplectic forms \(
\sigma_1,\sigma_2,\sigma_3 \) with for example \( \pi_C^*\sigma_3 =
i_C^*(U_1\hook U_2 \hook \Phi_C) \), where \( i_C\colon C(M)_0 =
\nu_C^{-1}(0) \hookrightarrow C(M) \) is the inclusion and \(
\pi_C\colon C(M)_0 \to N_C \) is the projection.  However,
\begin{equation*}
  \begin{split}
    \pi_C^*\sigma_3
    &= i_C^*(U_1 \hook U_2 \hook \Phi_C) \\
    &= i_C^*(s^3ds \wedge (U_1\hook U_2 \hook \phi) + \tfrac14
    U_1\hook U_2 \hook d\phi) \\
    &= i_C^*d(\tfrac14 s^4 \pi^* \eta_3),
  \end{split}
\end{equation*}
where \( \pi^*\eta_3 = i^*(U_1 \hook U_2 \hook \phi) \), with \(
i\colon M_0 = \nu^{-1}(0) \hookrightarrow M \) the inclusion and \(
\pi\colon M_0 \to N = M \mred T^3 \) the projection.  The fact
that \( \sigma_3 \) is non-degenerate on \( N_C = C(N) \) then
corresponds to \( \eta_3 \) being a contact structure on~\( N \).

In this way, we see that the reduction \( N = M \mred T^3 \) carries a
pointwise linearly independent triple \( (\eta_1,\eta_2,\eta_3) \) of
contact structures.  The condition that the \( \sigma_i \) are weakly
coherent corresponds to the requirements that the forms \( \eta_i
\wedge d\eta_i \) define the same orientation and that the symmetric
matrix with entries corresponding to \( \eta_i \wedge d\eta_j + \eta_j
\wedge d\eta_i \) is positive definite.  Thus one example is provided
by taking \( N = S^3 = \SU(2) \), with \( d\eta_1 = -2\eta_2 \wedge
\eta_3 \) etc.  Conversely the standard basis of one forms for \(
\SL(2,\bR) \) does not give a weakly coherent triple.

The remaining data for the \( \Spin(7) \)-geometry are the curvature
forms \( F_1,F_2,F_3 \) satisfying~\eqref{eq:QA}.  These forms are
invariant under the action of the Euler vector field on \( N_C \), so
they have the form \( F = d\log s \wedge a_i + b_i \), with \( a_i \in
\Omega^1(N) \) closed and \( b_i \in \Omega^2_{\bZ}(N) \).
Equation~\eqref{eq:QA} becomes
\begin{equation*}
  a_i \wedge d\eta_j + 4b_i \wedge \eta_j =
  a_j \wedge d\eta_i + 4b_j \wedge \eta_i,
  \qquad\text{for all \( i,j \).}  
\end{equation*}
Thus \( N \) carries three-contact forms and the closed forms \(
a_1,\dots,b_3 \).

\subsection{PSU(3)-structures}
\label{sec:psu3-structures}

A \( \PSU(3) \)-structure on an oriented eight-manifold~\( M \) is a
three-form \( \rho \in \Omega^3(M) \) pointwise modelled on~\( \rho_0
\) of equation~\eqref{eq:PSU3}.  These geometries were studied by
\textcite{Witt:triality,Witt:thesis}.  Such a form \( \rho \)
determines a metric~\( g \) and the five-form \( \Hodge\rho \).  A \(
\PSU(3) \)-structure is said to be \emph{harmonic} if \( d\rho = 0 \)
and \( d\Hodge\rho = 0 \).

The case when a harmonic \( \PSU(3) \)-structure admits a free
two-torus symmetry with multi-moment map \( \nu \) was discussed in
\cite{Madsen:mm-thesis} reinterpreting some results of Witt.  In this
case one has four two-forms given by
\begin{gather*}
  \omega_0 = -(d\nu)^\sharp \hook U_1 \hook U_2 \hook \Hodge\rho,
  \quad \omega_1 = U_1 \hook \rho, \\
  \quad \omega_2 = U_2 \hook \rho,
  \quad \omega_3 = U_1 \hook U_2 \hook \alpha^\sharp \hook
  \Hodge\rho, 
\end{gather*}
where \( \alpha \) is contraction of \( \rho \) by \( \omega_0 \).  At
a regular value \( t \) of \( \nu \), these forms induce two-forms~\(
\sigma_i \) on the reduction \( N = M \mred[\nu,t] T^2 \) and \(
\alpha \) induces a one-form \( a \).  Together \(
(a,\sigma_1,\sigma_2,\sigma_3) \) give \( N \) the structure of an \(
\SU(2) \)-manifold, see \textcite{Conti-S:Killing}, and a conformal
scaling of \( a \wedge \sigma_0 \) can be \( 2 \)-plectic.

\subsection{Homogeneous k-plectic manifolds}
\label{sec:homogeneous-k-plectic}

Suppose \( (M,\alpha) \) is a closed geometry of degree~\( r = k+1 \)
with a group~\( G \) of symmetries that acts transitively on~\( M \).
Then the equivariant map \( \Psi\colon M \to Z^r(\g) \) given
by~\eqref{eq:Psi} has image a single \( G \)-orbit in~\( Z^r(\g) \).
Conversely, we may use equation~\eqref{eq:Psi} to define closed
geometries that map to a given orbit \( G\cdot \Psi \subset Z^r(\g)
\), as follows.  Let \( K_\Psi \) be the connected subgroup of \( G \)
with Lie algebra \( \ker\Psi = \{\, \Lel X \in \g : \Lel X \hook \Psi
= 0 \, \} \).  Then for each closed subgroup \( H \) of \( G \)
containing \( K_\Psi \), equation~\eqref{eq:Psi} gives a well-defined
closed \( r \)-form \( \alpha \) on \( M = G/H \).

Now suppose that \( \Psi = \dP\beta \) for some \( \beta\in\Pg^* \).
If the map \( \dP \) is injective, then the orbits \( G\cdot\Psi \)
and \( G\cdot\beta \) are identified and the map \( \Psi\colon M\to
Z^r(\g) \) may now be interpreted as a map \( \nu\colon M\to \Pg^* \).
Injectivity of \( \dP \) is equivalent to the condition \( b_{r-1}(\g)
= 0 \) and the proof of \cref{thm:algebraic} shows that \( \nu \) is a
multi-moment map for the action of~\( G \).

\begin{theorem}
  \label{thm:orb1cgeo}
  Suppose \( G \) is a connected Lie group with \( b_k(\g)=0 \).
  Let \( \mathcal O = G \cdot \beta \subset \Pg^* \) be an orbit of \(
  G \) acting on the dual of the \( k \)th Lie kernel.  Then there
  are homogeneous closed geometries \( (G/H,\alpha) \), with \(
  \alpha\in\Omega^{k+1}(G/H) \) corresponding to \( \Psi = \dP\beta
  \), such that \( \mathcal O \) is the image of \( G/H \) under the
  (unique) multi-moment map~\( \nu \).

  The closed geometry may be realised on the orbit \( \mathcal O \)
  itself if and only if
  \begin{equation}
    \label{eq:stab}
    \stab_{\g}\beta = \ker(\dP\beta).
  \end{equation}
  In this situation, the orbit is \( k \)-plectic and \( \nu \) is
  simply the inclusion \( \mathcal O \hookrightarrow \Pg^* \).
\end{theorem}

\begin{proof}
  It only remains to prove the assertions of the last paragraph of the
  theorem.  We have \( \mathcal O = G/K \) with \( K = \stab_G\beta
  \), a closed subgroup of~\( G \).  Now equation~\eqref{eq:stab},
  shows that \( K \) has Lie algebra \( \ker(\dP\beta) \), so the
  component of the identity \( K^0 \) of \( K \) is \( K^0 = K_\Psi \)
  for \( \Psi = \dP\beta \).  In particular, \( \Psi \) vanishes on
  elements of \( \lk \) and induces a well-defined form on \(
  T_\beta\mathcal O = \g/\lk \).  The result now follows.
\end{proof}

\begin{remark}
  In the case when \( r=2 \), condition~\eqref{eq:stab} is automatic
  and we get the result of Kirillov-Kostant-Souriau that each orbit of
  \( \g^* \) is symplectic.
\end{remark}

\begin{example}
  Suppose \( G \) is a \( (k,k+1) \)-trivial Lie group.  Then, taking
  \( H=\{e\} \), we see that every \( \Psi \in Z^{k+1}(\g) \) gives
  rise to a closed geometry on \( G \) with multi-moment map whose
  image is diffeomorphic to the \( G \)-orbit of~\( \Psi \).
\end{example}

\section{Cohomology of Lie algebras}
\label{sec:cohomology}

Recall from \cref{sec:symmetries} that for a closed geometry of
degree~\( r \), multi-moment maps exist and are unique for any
symmetry group \( G \) which is \( (r-1,r) \)-trivial.  This condition
means that the cohomology groups~\( H^k(\g) \) of \( \g \) are trivial
in degrees \( k = r-1 \) and~\( r \).  While this is a concise
statement, it is not clear which, if any algebras, satisfy these
conditions.  In this section, we will discuss some techniques to gain
more information and show that there are in fact many such algebras.

Note first that the definitions~\eqref{eq:d-g} and~\eqref{eq:Hk-g}
give \( H^1(\g) = \ker d \leqslant \g^* \).  As \( d \) is dual to the
Lie bracket \( L = [\cdot,\cdot]\colon \Lambda^2\g \to \g \), the
vanishing of \( b_1(\g) = \dim H^1(\g) \) is equivalent to the
surjectivity of \( L \).  This says that \( b_1(\g) = 0 \) if and only
if \( \g \) is equal to its derived algebra \( \g' = [\g,\g] \).
Indeed \( b_1(\g) \) is exactly the codimension of \( \g' \) in \( \g
\).  Lie algebras with \( \g = \g' \) are called \emph{perfect}.  Any
semi-simple Lie algebra is perfect, but other examples may be
constructed as the semi-direct product \( \h \ltimes V \) of a
semi-simple algebra \( \h \) with a faithful representation~\( V \).
For example, the group of isometries of \( \bR^n \) with the standard
flat metric is perfect for each \( n\geqslant 3 \).

Interpretation of the vanishing of higher Betti numbers is more
complicated.  We gave some of the vanishing properties satisfied by
compact simple Lie groups in \cref{sec:symmetries}.  In particular,
these are \( (1,2) \)-trivial, leading to the usual existence and
uniqueness results for symplectic moment maps.  Furthermore any \(
(1,2) \)-trivial algebra is semi-simple.  The structure of the \(
(2,3) \)-trivial groups was described in~\cite{Madsen-S:multi-moment}
and classification results in small dimensions given in the same paper
and in~\cite{Madsen-S:2-3-trivial}.  In particular, we found that \(
(2,3) \)-trivial Lie algebras are always \emph{solvable}, meaning that
\( \g^m = \{ 0 \} \) for some \( m>0 \), where \( \g^m =
[\g^{m-1},\g^{m-1}] \) is the \( m \)th derived algebra of~\( \g \).

The general structure theory of Lie algebras says that any \( \g \)
has a maximal solvable ideal \( \lr \), the \emph{solvable radical},
and that the quotient \( {\g}/{\lr} \) is semi-simple.  In addition,
for any solvable Lie algebra \( \h \) the derived algebra \( \lk = \h'
\) is always \emph{nilpotent}, meaning that \( \lk_m = \{ 0 \} \) for
some \( m>0 \), where \( \lk_m = [\lk,\lk_{m-1}] \) and \( \lk_1 =
\lk' \).  In \cite{Madsen-S:multi-moment}, a result of
\textcite{Hochschild-Serre:Cohomology-Lie-algebras} relating the
cohomologies of \( \g \), \( \lr \) and \( {\g}/{\lr} \) was used
in~\cite{Madsen-S:multi-moment} to
prove:

\begin{lemma}
  \label{lem:b3}
  Any non-zero Lie algebra with \( b_3(\g) = 0 \) is solvable and thus
  has \( b_1(\g) \ne 0 \).  Such a \( \g \) is not nilpotent unless \(
  \g = \bR \) or \( \g = \bR^2 \).  \qed
\end{lemma}

The essential point is that the semi-simple algebra \( {\g}/{\lr} \)
has \( b_3 \) non-zero, and this feeds through to \( b_3(\g) \) if \(
\lr \ne \g \).  For a solvable algebra \( \g' \) is strictly smaller
than \( \g \), so \( b_1(\g) \ne 0 \).

The \( (2,3) \)-trivial algebras were then found to be exactly those
solvable \( \g \) such that the derived algebra \( \lk = \g' \) has
codimension~\( 1 \) in~\( \g \) and such that \( {\g}/{\lk} \) acts
invertibly on the cohomology groups~\( H^i(\lk) \), for \( i=1,2,3 \).

\subsection{(3,4)-trivial Lie algebras}

Let us work towards a description of \( (3,4) \)-trivial Lie algebras.
As a first result we consider direct sums of algebras.

\begin{proposition}
  \label{prop:product}
  A non-trivial direct sum \( \g = \h_1 + \h_2 \) has \( b_3(\g) = 0
  \) if and only if each summand~\( \h_i \) is \( (2,3) \)-trivial.
  Consequently, \( \g \) is product of at most two summands.

  The direct sum \( \g = \h_1 + \h_2 \) is \( (3,4) \)-trivial if and
  only if each summand is \( (2,3,4) \)-trivial.
\end{proposition}

\begin{proof}
  This is almost direct from the Künneth formula, which gives the
  following sum of positive terms:
  \begin{equation}
    \label{eq:b3}
    b_3(\g) = b_3(\h_1) + b_3(\h_2) + b_2(\h_1)b_1(\h_2) +
    b_1(\h_1)b_2(\h_2).
  \end{equation}
  Thus \( b_3(\g) = 0 \) immediately gives \( b_3(\h_i) = 0 \).
  However, by \cref{lem:b3} we known that \( b_1(\h_i) \ne 0 \), so
  the vanishing of \( b_3(\g) \) also gives \( b_2(\h_i) = 0 \).

  Now a similar argument, given in~\cite{Madsen-S:multi-moment}, shows
  that \( (2,3) \)-trivial algebras are not direct sums of smaller
  ideals.  Thus the summands \( \h_i \) are not direct sums and \( \g \)
  has at most two summands.

  In the second case, the Künneth formula gives
  \begin{equation}
    \label{eq:b4}
    b_4(\g) = b_4(\h_1) + b_4(\h_2) + b_2(\h_1)b_2(\h_2) +
    b_3(\h_1)b_1(\h_2) + b_1(\h_1)b_3(\h_2),
  \end{equation}
  so the extra condition \( b_4(\g) = 0 \) forces \( b_4(\h_i) = 0 \)
  too.

  The converse statements are immediate from \eqref{eq:b3} and
  \eqref{eq:b4}.
\end{proof}

To go further and study other cases we need to use stronger
techniques.  Our main tool will be the Hochschild-Serre spectral
sequence of a Lie algebra~\( \g \) with respect to an ideal~\( \lk \).
We will consider the case when the \( \lk \) contains the derived
algebra \( \g' \).  Then the quotient algebra \( \la = {\g}/{\lk} \)
is Abelian of rank at most \( b_1(\g) \).  This algebra acts on the
cohomology of \( \lk \) via
\begin{equation}
  \label{eq:a-action}
  \Lel A \cdot [\alpha] = [A \hook d\alpha],
  \qquad\text{for \( \Lel A \in \la =
  {\g}/{\lk} \), \( [\alpha] \in H^q(\lk) \).}
\end{equation}
Indeed the above formula defines an action of \( \g \) for which \(
\lk \) acts trivially.  Note that here \( d \) is the differential in
\( \g \); we will write \( d_0 \) for the differential in \( \lk \),
so \( \alpha \in \ker d_0 \) in~\eqref{eq:a-action}.  Moreover, in our
case with \( \la \) Abelian, this action induces the coboundary map \(
d_1 \) on the cochains
\begin{equation*}
  C^p(\la,H^q(\lk)) = \Lambda^p\la^* \otimes H^q(\lk)
\end{equation*}
via
\begin{equation*}
  (d_1f)(\Lel a) = \sum_{i=1}^{p+1} f(\Lel a_{\wedge i}),
\end{equation*}
for \( \Lel a \in \Lambda^{p+1}\la \).

The Hochschild-Serre spectral sequence
\cite{Hochschild-Serre:Cohomology-Lie-algebras} has \( E_2 \)-page
given by the cohomology of the operator \( d_1 \) above:
\begin{equation*}
  E_2^{p,q} = H^p(\la,H^q(\lk)).
\end{equation*}
Note that we have
\begin{equation*}
  H^0(\la,H^q(\lk)) =  \{\, b \in H^q(\lk) : \Lel A\cdot b = 0 \
  \text{for all \( \Lel A \in \lk \)} \,\} = H^q(\lk)^{\g}
\end{equation*}
the fixed-point set of the action of \( G \) on \(
H^q(\lk) \).  Also note that \( E_2^{p,q} = 0 \) for \( p > \dim \la
\).

Given the \( E_2 \)-page of the spectral sequence, the general theory
defines maps \( d_2\colon E_2^{p,q} \to E_2^{p+2,q-1} \) induced by
the exterior derivative \( d \) in \( \g \) and sets \( E_3^{p,q} \)
to be the corresponding cohomology group.  More generally, \(
d_r\colon E_r^{p,q} \to E_r^{p+r,q-r+1} \) and so the spectral
sequence stabilises at level \( r = \dim \la \).  So \( E_\infty^{p,q}
= E_{\dim a}^{p,q} \) and one then has
\begin{equation*}
  H^k(\g) \cong \bigoplus_{p+q=k} E_\infty^{p,q}.
\end{equation*}
Note that if we choose a linear splitting of the exact sequence
\begin{equation*}
  0 \to \lk \to \g \to \la \to 0
\end{equation*}
then the image of \( W = \lk^* \) in \( \g^* \) has \( dW \subset
\Lambda^2W + \la^*\wedge W + \Lambda^2\la^* \).  In particular, the
differential of \( \g \) on \( E_0^{p,q} \) has components in \(
E_0^{p,q+1} \), \( E_0^{p+1,q} \) and \( E_0^{p+2,q-1} \).  Thus this
is more general than the spectral sequence of a bicomplex.

We are now ready to state our first characterisation result, the proof
of which will be given after some discussion of consequences.

\begin{theorem}
  \label{thm:3-4-structure-1}
  A Lie algebra \( \g \) is \( (3,4) \)-trivial if and only if \( \g
  \) is solvable and for any codimension one ideal \( \lk \)
  containing \( \g' \) one has \( H^i(\lk)^{\g} = 0 \) for \( i =
  2,3,4 \).  Here \( H^i(\lk)^{\g} \) is the part of the cohomology of
  \( \lk \) that is invariant under the action of \( \g \).
\end{theorem}

This result already gives a number of examples of \( (3,4) \)-trivial
algebras.

\begin{example}
  \label{ex:Abelian}
  Let \( \lk \) be an Abelian algebra \( \bR^m \) of dimension \( m
  \).  The differential in \( \lk \) is zero and \( H^k(\lk) =
  \Lambda^k\lk^* \) for each \( k \).  Let \( T \) be a diagonalisable
  operator acting by \( T(\Lel K_i) = \lambda_i\Lel K_i \), where \(
  \Lel K_1,\dots,\Lel K_m \) is a basis of eigenvectors.  Write \(
  \kappa_1,\dots,\kappa_m \) for the dual basis of \( \lk^* \).  Then
  \( T \) acts on \( H^k(\lk) \) with eigenvalue \( \lambda_{i_1} +
  \dots + \lambda_{i_k} \) on \( \kappa_{i_1} \wedge \dots \wedge
  \kappa_{i_k} \).  We thus have that \( H^i(\lk)^T = 0 \) for \(
  i=2,3,4 \) if and only if \( \lambda_i \ne
  -\lambda_j,-\lambda_j-\lambda_k,-\lambda_j-\lambda_k-\lambda_\ell \)
  whenever \( i,j,k,\ell \) are distinct.  This then gives a \( (3,4)
  \)-trivial algebra \( \g \) defined by
  \begin{equation*}
    \g = \bR\Lel A + \lk ,
  \end{equation*}
  where \( [\Lel A,\Lel K_i] = T\Lel K_i = \lambda_i\Lel K_i \).  Note
  that if some \( \lambda_i \) is~\( 0 \), then this occurs for only
  one index~\( i \) and \( \g \) splits as a product \( \bR + \h \),
  with \( \Lel T \) acting invertibly on \( \h' \cong \bR^{m-1} \).

  Some concrete \( (3,4) \)-trivial non-product examples are given as
  follows.  First consider \( \lk = \bR^2 \).  Then \( \g =
  (0,12,\mu.13) \), \( \mu \ne 0,-1 \), is \( (3,4) \)-trivial.  Here
  the notation means that \( \g^* \) has a basis \( e_1,e_2,e_3 \)
  with \( de_1 = 0 \), \( de_2 = e_1 \wedge e_2 \) and \( de_3 = \mu
  e_1 \wedge e_3 \).  For a higher-dimensional case with \( \lk =
  \bR^4 \), another example is provided by the algebras \( \g =
  (0,12,13,14,\mu.15) \) with \( \mu \ne 0,-1,-2,-3 \).  For
  general~\( m \), by taking \( \lambda_i \) strictly positive for
  each \( i \), we obtained \( (3,4) \)-trivial algebras in all
  dimensions.
\end{example}

\begin{example}
  \label{ex:positive}
  Suppose \( \lk \) is \emph{positively graded}, meaning that as a
  vector space \( \lk = \lk_1 \oplus \lk_2 \oplus \dots \oplus \lk_r
  \) with \( [\lk_i,\lk_j] \subset \lk_{i+j} \) for all \( i,j \).
  Such a \( \lk \) is necessarily nilpotent.  Choosing a linear
  operator \( T \) on \( \lk \), we may obtain a Lie algebra \( \g =
  \bR\Lel A + \lk \) with \( [\Lel A,\Lel K] = T(\Lel K) \) if only if
  \( T[\Lel K_1,\Lel K_2] = [T(\Lel K_1),\Lel K_2] + [\Lel K_1,T(\Lel
  K_2)] \) for all \( \Lel K_i \in \lk \).  In the positively graded
  situation we may thus take \( T(\Lel K) = j\Lel K \), for \( \Lel K
  \in \lk_j \).  

  Every nilpotent algebra of dimension at most \( 6 \) admits a
  positive grading (see \cite{Madsen-S:2-3-trivial} for concrete
  gradings), so each of these nilpotent Lie algebras may be realised
  as the derived algebra of a \( (3,4) \)-trivial algebra.  As a
  concrete example, the seven-dimensional algebra
  \begin{equation*}
    (0,12,3.13,4.14+23,5.15+24,6.16+25,7.17+34+26)
  \end{equation*}
  obtained from \( (0,0,12,13,14,23+15) \) via the grading given by the
  weights \( (1,3,4,5,6,7) \) is \( (3,4) \)-trivial.
\end{example}

\Cref{thm:3-4-structure-1} has the following consequence, which
completes the characterisation \cref{prop:product} of \( (3,4)
\)-trivial algebras that are direct sums.

\begin{corollary}
  Let \( \g \) be a Lie algebra.  Write \( \lk = \g' \) for the
  derived algebra and \( \la = \g / \lk \).  Then \( \g \) is \(
  (2,3,4) \)-trivial if and only if \( \g \) is solvable, \( \la \)
  has dimension \( 1 \) and \( H^i(\lk)^{\g} = \{0\} \), for \(
  i=1,2,3,4 \).
\end{corollary}

\begin{proof}
  By~\cite{Madsen-S:multi-moment} any \( (2,3) \)-trivial algebra has
  \( \codim \g' = 1 \) and \( H^i(\g')^{\g} = \{0\} \), for \( i=1,2,3
  \).  Combining this with \cref{thm:3-4-structure-1} applied to \(
  \lk = \g' \) gives the result.
\end{proof}

\begin{proof}[of \cref{thm:3-4-structure-1}]
  By \cref{lem:b3} we know \( \g \) is solvable, so \( \g' \) has
  codimension at least one.  Now with respect to a codimension one
  ideal \( \lk \), the \( E_1 \)-page of the spectral sequence has
  row~\( q \) given by
  \begin{equation*}
    \begin{tikzpicture}
      \matrix (Eone) [matrix of math nodes,nodes in empty cells,column
      sep=4ex,nodes={minimum width=6ex}] {
      q+1 & \cdots \\
      q & H^q(\lk) &\la^*\otimes H^q(\lk)& 0 &0&\cdots \\
      q-1 & \cdots \\
      & 0 & 1 & 2 & 3 & \strut\cdots \\
      }; \draw (Eone-1-1.north east) |- (Eone-4-6.north east);
      \draw[->] (Eone-2-2) -- (Eone-2-3) node[above,midway,scale=0.7]
      {\( d_1 \)};
    \end{tikzpicture}
  \end{equation*}
  For the \( E_2 \)-terms we thus have \( E_2^{0,q} = \ker d_1 =
  H^q(\lk)^{\g} \) and \( E_2^{1,q} = \coker d_1 \).  But \( \dim\la =
  1 \), so \( d_1 \) is a map between vector spaces of the same
  dimension.  This implies \( \coker d_1 \) and \( \ker d_1 \) have
  the same dimension and so we may identify \( E_2^{1,q} \) with \(
  H^q(\lk)^{\g} \), non canonically.

  The \( E_2 \)-page of the spectral sequence is thus
  \begin{equation*}
    \begin{tikzpicture}
      \matrix (Etwo) [matrix of math nodes,nodes in empty cells,column
      sep=2ex,nodes={minimum width=6ex}] {
      \vdots & \cdots \\
      4 & H^4(\lk)^{\g} &H^4(\lk)^{\g}& 0 &0&\cdots \\
      3 & H^3(\lk)^{\g} &H^3(\lk)^{\g}& 0 &0&\cdots \\
      2 & H^2(\lk)^{\g} &H^2(\lk)^{\g}& 0 &0&\cdots \\
      1 & H^1(\lk)^{\g} &H^1(\lk)^{\g}& 0 &0&\cdots \\
      0 & \strut\bR & \bR& 0 &0&\cdots \\
      & 0 & 1 & 2 & 3 & \strut\cdots \\
      }; \draw (Etwo-1-1.north east) |- (Etwo-7-6.north east);
    \end{tikzpicture}
  \end{equation*}
  The spectral sequence degenerates at the \( E_2 \)-term and we
  conclude that
  \begin{equation*}
    H^3(\g) \cong H^3(\lk)^{\g} + H^2(\lk)^{\g},\quad H^4(\g) \cong
    H^4(\lk)^{\g} + H^3(\lk)^{\g},
  \end{equation*}
  from which the result follows.
\end{proof}

In the case when \( \n \) is a nilpotent Lie algebra, one may use the
above spectral sequence to prove Dixmier's result
\cite{Dixmier:nil-cohomology} that \( b_k(\n) \geqslant 2 \) for each
\( 0< k < \dim\n \), since in this situation \( \g = \n \) acts
nilpotently on \( H^q(\lk) \) and so \( H^q(\lk)^{\g} \) is non-zero
if \( H^q(\lk) \) is non-zero.  A refinement of the above argument was
also used by \textcite{Cairns-J:Betti} to improve Dixmier's bounds.

\begin{proposition}
  \label{prop:3-4-trivial}
  A Lie algebra \( \g \) with derived algebra \( \g' \) of codimension
  at least two is \( (3,4) \)-trivial if and only if for each ideal \(
  \lk \) of \( \g \) containing \( \g' \) one has \( H^i(\lk)^{\g} = 0
  \) for \( i = 1,2,3,4 \).
\end{proposition}

The reason for giving this result in addition to
\cref{thm:3-4-structure-1} is that we will see later that we can
often assume that \( \codim \g' \leqslant 2 \) and therefore take \(
\lk = \g' \) to be nilpotent.

\begin{example}
  \label{ex:Abelian-2}
  Suppose \( \lk \) is Abelian as in \cref{ex:Abelian}.  We may now
  consider two commuting diagonalisable operators \( T_1 \) and \( T_2
  \) acting on~\( \lk \).  Then \( \lk \) splits as a direct sum of
  common eigenspaces of \( T_1 \) and \( T_2 \).  We wish to consider
  when the Lie algebra given by
  \begin{equation}
    \label{eq:g-codim-2}
    \begin{gathered}
      \g = \Span{\Lel A_1,\Lel A_2} + \lk,\\
      [s\Lel A_1 + t\Lel A_2,\Lel K] = (sT_1+tT_2)(\Lel K),
      \quad\text{for \( s,t \in \bR \), \( \Lel K \in \lk \)},
    \end{gathered}
  \end{equation}
  is \( (3,4) \)-trivial.

  The cohomology conditions of \cref{prop:3-4-trivial} are satisfied
  if \( T_1 \) and \( T_2 \) have no common \( 0 \)-eigenspace on \(
  \Lambda^i\lk^* \) for \( i=1,2,3,4 \).  This is a vanishing
  condition on finitely many linear combinations of the eigenvalues of
  \( T_1 \) and of \( T_2 \).  When it is satisfied we may thus find a
  generic linear combination \( T = sT_1 + tT_2 \) such that the
  action of \( T \) on these \( \Lambda^i\lk^* \) has trivial kernel.
  Taking \( T_1 = T \), we are now free to have \( T_2 \) any linear
  transformation of \( \lk \) that commutes with \( T_1 \), and need
  not assume that \( T_2 \) is diagonalisable.

  Thus for example when \( \lk \cong \bR^2 \), we see that the algebra
  \begin{equation*}
    (0,0,13+24,14)
  \end{equation*}
  is \( (3,4) \)-trivial.
\end{example}

\begin{example}
  If \( \lk \) is positively graded, as in \cref{ex:positive}, we may
  take \( T_1 = T \) as in that example and let \( \Lel A_2 \) be any
  derivation~\( T_2 \) of \( \lk \) commuting with \( T_1 \).  Then
  \eqref{eq:g-codim-2} defines a \( (3,4) \)-trivial algebra.  A
  simple example is obtained by taking \( \lk \) to be the Heisenberg
  algebra \( (0,0,12) \).  This has a positive grading with weights \(
  (1,1,2) \).  A second derivation~\( T_2 \) acts on \( \lk^* \) by \(
  e_1 \mapsto e_2 \), \( e_2 \mapsto -e_1 \), \( e_3 \mapsto 0 \) and
  this commutes with \( T_1 \).  We thus conclude that
  \begin{equation*}
    (0,0,13 + 24,14 - 23,2.15)
  \end{equation*}
  is \( (3,4) \)-trivial.  Another choice for \( T_2 \) is the
  nilpotent transformation \( e_1 \mapsto e_2 \), \( e_2 \mapsto 0 \),
  \( e_3 \mapsto 0 \) of \( \lk^* \), which gives the \( (3,4)
  \)-trivial algebra
  \begin{equation*}
    (0,0,13+24,14,2.15).
  \end{equation*}
\end{example}

\begin{proof}[of \cref{prop:3-4-trivial}]
  This time \( \la^*\cong \bR^2 \) and \( \Lambda^2\la^*\cong \bR \),
  so the \( E_1 \)-page of the spectral sequence has row~\( q \)
  isomorphic to
  \begin{equation*}
    \begin{tikzpicture}
      \matrix (Eone) [matrix of math nodes,nodes in empty cells,column
      sep=4ex,nodes={minimum width=6ex}] {
      q+1 & \cdots \\
      q & H^q(\lk) & 2H^q(\lk)& H^q(\lk) &0&\cdots \\
      q-1 & \cdots \\
      & 0 & 1 & 2 & 3 & \strut\cdots \\
      }; \draw (Eone-1-1.north east) |- (Eone-4-6.north east);
      \draw[->] (Eone-2-2) -- (Eone-2-3) node[above,midway,scale=0.7]
      {\( d_1 \)}; \draw[->] (Eone-2-3) -- (Eone-2-4)
      node[above,midway,scale=0.7] {\( d_1 \)};
    \end{tikzpicture}
  \end{equation*}
  Since the maps \( d_1 \) on the bottom row \( q=0 \) are zero, the
  relevant part of the \( E_2 \)-page is
  \begin{equation*}
    \begin{tikzpicture}
      \matrix (Etwo) [matrix of math nodes,nodes in empty cells,column
      sep=4ex,nodes={minimum width=6ex}] {
      \vdots & \cdots \\
      4 & H^4(\lk)^{\g} & E_2^{1,4}& E_2^{2,4} &0&\cdots \\
      3 & H^3(\lk)^{\g} & E_2^{1,3}& E_2^{2,3} &0&\cdots \\
      2 & H^2(\lk)^{\g} & E_2^{1,2}& E_2^{2,2} &0&\cdots \\
      1 & H^1(\lk)^{\g} & E_2^{1,1}& E_2^{2,1} &0&\cdots \\
      0 & \strut\bR & \bR^2& \bR &0&\cdots \\
      & 0 & 1 & 2 & 3 & \strut\cdots \\
      }; \draw (Etwo-1-1.north east) |- (Etwo-7-6.north east);
      \draw[->] (Etwo-2-2) -- (Etwo-3-4); \draw[->] (Etwo-3-2) --
      (Etwo-4-4); \draw[->] (Etwo-4-2) -- (Etwo-5-4);
    \end{tikzpicture}
  \end{equation*}
  The maps \( d_2 \) from and to the middle column \( E_2^{1,p} \) are
  zero, and the spectral sequence degenerates at the \( E_3 \)-level,
  so
  \begin{gather*}
    H^3(\g) \cong E_3^{0,3} + E_2^{1,2} + E_3^{2,1},\\
    H^4(\g) \cong E_3^{0,4} + E_2^{1,3} + E_3^{2,2}.
  \end{gather*}
  Thus \( b_3(\g) = 0 = b_4(\g) \) implies that \( E_2^{1,2} = 0 =
  E_2^{1,3} \).  This says that the middle cohomologies of the \( E_1
  \)-page on rows \( q=2 \) and \( q=3 \) are zero.  However, for a
  sequence \( V \overset{\alpha}{\rightarrow} 2V
  \overset{\beta}{\rightarrow} V \), vanishing of the middle
  cohomology implies that \( \alpha \) is injective and \( \beta \) is
  surjective.  Thus the cohomologies at each end of these rows are also
  zero.  In particular,
  \begin{equation*}
    H^2(\lk)^{\g} = 0 \quad\text{and}\quad H^3(\lk)^{\g} = 0.
  \end{equation*}
  The \( E_2 \)-page is now
  \begin{equation*}
    \begin{tikzpicture}
      \matrix (Etwo) [matrix of math nodes,nodes in empty cells,column
      sep=4ex,nodes={minimum width=6ex}] {
      \vdots & \cdots \\
      4 & H^4(\lk)^{\g} & E_2^{1,4}& E_2^{2,4} &0&\cdots \\
      3 & 0 & 0 & 0 &0&\cdots \\
      2 & 0 & 0 & 0 &0&\cdots \\
      1 & H^1(\lk)^{\g} & E_2^{1,1}& E_2^{2,1} &0&\cdots \\
      0 & \strut\bR & \bR^2& \bR &0&\cdots \\
      & 0 & 1 & 2 & 3 & \strut\cdots \\
      }; \draw (Etwo-1-1.north east) |- (Etwo-7-6.north east);
    \end{tikzpicture}
  \end{equation*}
  all the \( d_2 \) maps marked in the first picture of the \( E_2
  \)-page are zero and
  \begin{equation*}
    H^3(\g) \cong E_2^{2,1},\qquad
    H^4(\g) \cong H^4(\lk)^{\g}.
  \end{equation*}
  Thus \( (3,4) \)-trivial implies
  \begin{equation*}
    H^4(\lk)^{\g} = 0.
  \end{equation*}

  To prove the necessity of the cohomology conditions it only remains
  to show that \( E_2^{2,1} = 0 \) if and only if \( H^1(\lk)^{\g} = 0
  \).  In fact to get sufficiency we will show that \( H^q(\lk)^{\g} =
  0 \) is equivalent to \( E_2^{2,q} = 0 \).

  Write \( V = H^q(\lk) \otimes \bC = H^q(\lk)_\bC \) and let \( \Lel
  A, \Lel B \) be a basis for \( \la \).  The space \( (E_2^{2,q})_\bC
  \) is the cokernel of \( d_1 \colon V\times V \to V \) with \(
  d_1(f_1,f_2) = A f_2 - B f_1 \), where \( A \), \( B \) are the
  linear operators of the action of \( \Lel A \) and \( \Lel B \) on
  \( V \).  Thus \( (E_2^{2,q})_\bC \) is \( V/(\im A+\im B)\),
  whereas \( H^q(\lk)^{\g}_\bC \) is \( \ker A \cap \ker B \).

  Now \( A \) and \( B \) commute, so each preserves the generalised
  eigenspaces of the other.  Decompose \( V \) as a direct sum of the
  common generalised eigenspaces \( V = \bigoplus E(\lambda) \), where
  \( E(\lambda_1,\lambda_2) = E_A(\lambda_1) \cap E_B(\lambda_2) =
  \ker((A-\lambda_1)^n) \cap \ker((B-\lambda_2)^n) \), \( n = \dim V
  \).  We see that \( A \) or \( B \) is invertible on each space \(
  E(\lambda) \) with \( \lambda \ne (0,0) \).  Thus \( \ker A \cap
  \ker B \) is a subspace of \( E(0) \) and \( V/(\im A+\im B) \) is a
  quotient space of \( E(0) \).

  On \( E(0) \), the operators \( A \) and \( B \) are nilpotent and
  commute.  They thus generate a nilpotent Lie algebra and by Engel's
  Theorem, there is a basis \( v_1,\dots,v_m \) for \( E(0) \) such
  that \( A \) and \( B \) are upper triangular.  In particular \( \im
  A + \im B \leqslant \Span{v_1,\dots,v_{m-1}} \) and \( v_1 \in \ker
  A \cap \ker B \).  Thus \( V/(\im A+\im B) \) or \( \ker A\cap\ker B
  \) is zero, if and only if \( E(0) \) is zero.  This proves that \(
  H^q(\lk)^{\g} = 0 \) if and only if \( E_2^{2,q} = 0 \).

  We thus have that the given cohomology conditions are necessary.  To
  see that they are sufficient, note that the vanishing of \(
  H^q(\lk)^{\g} \) implies that the sequence on the \( q \)-row of the
  \( E_1 \)-page has the form \( V \overset{\alpha}{\rightarrow} 2V
  \overset{\beta}{\rightarrow} V \) with \( \alpha \) injective and \(
  \beta \) surjective.  It follows that the middle cohomology is also
  zero and that \( E_2^{p,q} \) is zero for all \( 1 \leqslant p
  \leqslant 4 \).  This implies that \( b_3(\g) = 0 = b_4(\g) \) as
  required.
\end{proof}

\Cref{thm:3-4-structure-1,prop:3-4-trivial} give a full structural
description of the \( (3,4) \)-trivial algebras with \( \g' \) of
codimension at most two.  The next result gives a couple of conditions
under which this codimension assumption is guaranteed.

\begin{proposition}
  \label{prop:codim-2}
  A Lie algebra \( \g \) with \( b_3(\g) = 0 \) has derived algebra \(
  \lk = \g' \) of codimension at most two if either
  \begin{compactenum}
  \item \( \g \) is split, so \( \g \) is the semi-direct product of
    \( \la = {\g}/{\lk} \) and \( \lk \), or
  \item \( H^1(\lk)^{\g} = \{0\} = H^2(\lk)^{\g} \).
  \end{compactenum}
\end{proposition}

\begin{proof}
  When \( \g \) is split, we have \( \g = \la + \lk \), \( [\la,\la] =
  0 \), and so the differential \( d \) of \( \g \) satisfies \(
  d\la^* = 0 \) and \( d\lk^* \subset \Lambda^2\lk^* +
  \la^*\wedge\lk^* \).  In particular, no element of \( \Lambda^k\la^*
  \) is exact, and there is an injection \( \Lambda^k\la^* \to H^k(\g)
  \) for each~\( k \).  Thus \( b_3(\g) = 0 \) implies \(
  \Lambda^3\la^* = 0 \) and \( \la \) has dimension at most two, as
  required.

  For the second case of the \namecref{prop:codim-2}, we consider the
  spectral sequence.  Choose a basis \( \Lel A_1,\dots,\Lel A_r \)
  for~\( \la \). The proof of \cref{prop:3-4-trivial} shows that \(
  H^q(\lk)^{\g} = \{0\} \) is equivalent to common generalised
  eigenspace \( E(0) = \bigcap_{i=1}^r \ker(A_i^n) \) of the induced
  operators \( A_i \) on \( H^q(\lk) \) being zero.  This implies that
  some linear combination \( \Lel A \) of the \( \Lel A_i \) acts
  invertibly on \( H^q(\lk) \).  We may thus choose our basis \( \Lel
  A_i \) so that \( A_1 \) acts invertibly on the given \( H^q(\lk)
  \).

  Now consider the cohomology of
  \begin{equation*}
 H^q(\lk)
  \overset{d_1}{\longrightarrow} \la^*\otimes H^q(\lk)
  \overset{d_1}{\longrightarrow} \Lambda^2\la^*\otimes H^q(\lk).
\end{equation*}
The first map is given by \( v \mapsto (A_1v,\dots,A_rv) \), the
second by \( (f_1,\dots,f_r) \mapsto (A_if_j-A_jf_i) \).  For \(
\mathbf f = (f_1,\dots,f_r) \in \ker d_1 \subset \la^*\otimes H^q(\lk)
\), we may write \( v = A_1^{-1}f_1 \).  Then \( A_1f_j-A_jf_1 = 0 \),
implies that \( f_j = A_1^{-1}(A_jf_1) = A_jv \) and that \( \mathbf f
\in \im d_1 \).  Thus the \( d_1 \)-cohomology vanishes at the second
step.

  Under the hypotheses of the \namecref{prop:codim-2}, the \( E_2
  \)-page of the spectral sequence for \( \g \) with respect the
  derived algebra \( \lk \) is now
  \begin{equation*}
    \begin{tikzpicture}
      \matrix (Etwo) [matrix of math nodes,nodes in empty cells,column
      sep=4ex,nodes={minimum width=6ex}] {
      \vdots & \cdots \\
      3 & H^3(\lk)^{\g} & E_2^{1,3}& E_2^{2,3} & E_2^{3,3} &\cdots \\
      2 & 0 & 0 & E_2^{2,2} & E_2^{3,2} &\cdots \\
      1 & 0 & 0 & E_2^{2,1} & E_2^{3,1} &\cdots \\
      0 & \strut\bR & \la^*& \Lambda^2\la^* & \Lambda^3\la^*&\cdots \\
      & 0 & 1 & 2 & 3 & \strut\cdots \\
      }; \draw (Etwo-1-1.north east) |- (Etwo-7-6.north east);
      \draw[->] (Etwo-4-3) -- (Etwo-5-5); \draw[dashed,->]
      (Etwo-3-2.east) parabola (Etwo-5-5.north west);
    \end{tikzpicture}
  \end{equation*}
  with the solid arrow to \( \Lambda^3\la^* \) representing \( d_2 \)
  and the dashed arrow indicating \( d_3 \).  This gives that \(
  \Lambda^3\la^* = E_2^{3,0} = E_3^{3,0} = E_\infty^{3,0} \) is a
  summand of \( H^3(\g) \).  Hence \( b_3(\g) \geqslant
  \dim\Lambda^3\la^* \) and \( b_3(\g) = 0 \) implies that \( \dim\la
  \leqslant 2 \), as required.
\end{proof}

\begin{remark}
  Note that the condition \( H^1(\lk)^{\g} = \{0\} \) does not appear
  in \cref{thm:3-4-structure-1}.  However, when \( \lk = \g' \) is of
  codimension one, this space is automatically zero.  Indeed, suppose
  that \( \g = \bR\Lel A + \lk \) with \( \lk \) an ideal.  A non-zero
  element of \( H^1(\lk)^{\g} \) is an element \( \gamma \in \lk^* \),
  such that \( d_0\gamma = 0 \) and \( \Lel A\cdot \gamma = 0 \).
  The first condition implies that \( \gamma \) annihilates \( \lk'
  \), the second says that \( \gamma([\Lel A,\Lel K]) = 0 \) for each
  \( \Lel K \in \lk \).  Thus \( \gamma \) annihilates \( \g' \), but
  so does \( \la^* \), so \( \g' \) is at least codimension two.
\end{remark}

The Betti numbers of Lie algebras of dimension at most \( 6 \) are
given in
\textcite{Freibert-S-H:decomposable,Freibert-S-H:indecomposable}.  One
may thus use their tables to read off which algebras of dimension \( 4
\), \( 5 \) or \( 6 \) are \( (3,4) \)-trivial.
\Cref{tab:34Freibert-S-H} summarises the resulting sets of Betti
numbers that occur.  The tables confirm that \( \g' \) is of
codimension at most two in these cases.  When the codimension equals
two, i.e., \( b_1(\g) = 2 \), one sees that \( b_n(\g) = 0 \).  In
other words, these examples are not unimodular.  The next example
shows that this is no longer true in higher dimensions.

\begin{table}[htp]
  \centering
  \begin{tabular}{CC}
    \toprule
    n=\dim\g & (b_1,\dots,b_n)                               \\
    \hline
    4        & (1,0,0,0),\,(2,1,0,0)                         \\
    5        & (1,0,0,0,0),\,(2,1,0,0,0)                     \\
    6        & (1,0,0,0,0,0),\,(1,0,0,0,1,1),\,(2,1,0,0,0,0) \\
    \bottomrule
  \end{tabular}
  \caption{The sets of Betti numbers of the \((3,4)\)-trivial Lie
  algebras~\( \g \) appearing in the low-dimensional classifications of
  \cite{Freibert-S-H:decomposable,Freibert-S-H:indecomposable}.}
  \label{tab:34Freibert-S-H}
\end{table}

\begin{example}
  Consider an \( (m+2) \)-dimensional Lie algebra of the form as
  in~\cref{ex:Abelian-2} with \( T_1 \) and \( T_2 \) commuting and
  diagonalisable.  The proof of \cref{prop:3-4-trivial} shows that we
  may assume that \( T_1 \) is invertible.  The action of \( T_i \) on
  \( H^m(\lk) \) is simply the trace of the action~\( \lk^* \).  When
  \( m \geqslant 5 \), we may ensure that the cohomology conditions of
  \cref{prop:3-4-trivial} are satisfied and that each \( T_i \) acts
  trivially on \( H^m(\lk) \), for example by giving both \( T_i \)
  strictly positive eigenvalues on some \( (m-1) \)-dimensional
  subspace of~\( \lk \) and requiring \( T_i \) to be trace-free.  Now
  this trace-free condition ensures that the maps \( d_1\colon
  E_1^{p,m} \to E_1^{p+1,m} \) of the Hochschild-Serre spectral
  sequence for \( \g \) with respect to the ideal \( \lk \) are
  identically zero, so \( H^n(\g) = E_2^{2,m} = E_1^{2,m} = H^m(\lk) =
  \bR \), showing that \( \g \) is unimodular.  A concrete non-product
  unimodular \( (3,4) \)-trivial algebra of dimension \( n = m+2
  \geqslant 7 \) with \( b_1 = 2 \) is given by
  \begin{equation*}
    (0,0,13+23,14,15,\dots,(1-m).1n -2n).
  \end{equation*}
\end{example}

\printbibliography

\bigskip

\begin{small}
  \parindent0pt\parskip\baselineskip

  T.B.Madsen \& A.F.Swann

  Department of Mathematical Sciences, University of Aarhus, Ny
  Munkegade 118, Bldg 1530, DK-8000 Aarhus C, Denmark.

  \textit{and}

  CP\textsuperscript3-Origins, Centre of Excellence for Particle
  Physics Phenomenology, University of Southern Denmark, Campusvej 55,
  DK-5230 Odense M, Denmark.

  \textit{E-mail}: \url{tbmadsen@imf.au.dk}, \url{swann@imf.au.dk}
\end{small}

\end{document}